\documentclass[11pt]{amsart}

\usepackage{fullpage}
\usepackage[utf8]{inputenc}
\usepackage[T1]{fontenc}
\usepackage{amsmath,amsfonts,amssymb,amsthm,mathtools}
\usepackage{mathrsfs,dsfont}
\usepackage{stmaryrd}
\usepackage{enumerate}

\usepackage{hyperref}

\usepackage{enumerate}

\usepackage{color}

\newcommand{\E}{\mathbb{E}}
\newcommand{\R}{\mathbb{R}}
\newcommand{\N}{\mathbb{N}}

\newtheorem{theo}{Theorem}[section]

\newtheorem{rem}[theo]{Remark}

\newtheorem{propo}[theo]{Proposition}
\newtheorem{lemma}[theo]{Lemma}
\newtheorem{defi}[theo]{Definition}

\newtheorem{hyp}[theo]{Assumption}

\newcommand{\Ly}{\mathbf{L}}

\newcommand{\Xb}{\overline{X}}
\newcommand{\Fb}{\overline{F}}
\newcommand{\ub}{\overline{u}}
\newcommand{\Lb}{\overline{\mathcal{L}}}

\newcommand{\gamy}{\gamma_{\rm max}}
\newcommand{\aly}{\alpha_{\rm max}}

\newcommand{\ir}{4}

\begin{document}

\title{Orders of convergence in the averaging principle for SPDE{\tiny s}: the case of a stochastically forced slow component}

\author{Charles-Edouard Br\'ehier}
\address{Univ Lyon, CNRS, Universit\'e Claude Bernard Lyon 1, UMR5208, Institut Camille Jordan, F-69622 Villeurbanne, France}
\email{brehier@math.univ-lyon1.fr}

\keywords{Stochastic Partial Differential Equations, Averaging Principle, Poisson equation in infinite dimension, Heterogeneous Multiscale Method, strong and weak error estimates}

\date{}

\begin{abstract}
This article is devoted to the analysis of semilinear, parabolic, Stochastic Partial Differential Equations, with slow and fast time scales. Asymptotically, an averaging principle holds: the slow component converges to the solution of another semilinear, parabolic, SPDE, where the nonlinearity is averaged with respect to the invariant distribution of the fast process.

We exhibit orders of convergence, in both strong and weak senses, in two relevant situations, depending on the spatial regularity of the fast process and on the covariance of the Wiener noise in the slow equation. In a very regular case, strong and weak orders are equal to $\frac12$ and $1$. In a less regular case, the weak order is also twice the strong order.

This study extends previous results concerning weak rates of convergence, where either no stochastic forcing term was included in the slow equation, or the covariance of the noise was extremely regular.

An efficient numerical scheme, based on Heterogeneous Multiscale Methods, is briefly discussed.
\end{abstract}

\maketitle

\section{Introduction}

Systems with multiple time scales, and possibly stochastic forcing terms, appear in all fields of modern science, at fundamental and applied levels, for instance in physics, chemistry, biology, engineering, etc... Understanding how properties at micro-scales transfer to macro-scales, and the design of efficient numerical schemes, are still challenging issues. In the mathematical literature, powerful limiting procedures have been developed, e.g. averaging and homogenization techniques are available. We refer for instance to~\cite{FouqueGarnierPapanicolaouSolna,FreidlinWentzell,Kuehn,PavliotisStuart} for monographs devoted to the study of multiscale stochastic systems.

The averaging principle can be interpreted as a law of large numbers, in cases where a slow component is driven by an equation with coefficients depending on a fast component, which is an ergodic stochastic process: when the separation of time scales goes to infinity, the slow component converges to the solution of an averaged equation, where coefficient have been averaged out with respect to some invariant probability distribution for the fast component. For results concerning the averaging principle for Stochastic Differential Equations (SDEs), we refer to the pioneering work~\cite{Khasminskii:1968}, and to the following extensions since then, e.g.~\cite{KhasminskiiYin:2004},~\cite{KhasminskiiYin:2005},~\cite{Liu:2010},~\cite{Veretennikov:1990} (this list is not exhaustive).

In this article, the aim is to study the averaging principle for a class of parabolic, semilinear, Stochastic Partial Differential Equations (SPDEs)
\begin{equation*}
\begin{aligned}
\frac{\partial x^\epsilon(t,\xi)}{\partial t}&=\Delta x^\epsilon(t,\xi)+f\bigl(x^\epsilon(t,\xi),y^\epsilon(t,xi)\bigr)+\dot{W}(t,\xi),\quad t\ge 0,~\xi\in\mathcal{D},\\
x^{\epsilon}(t,\cdot)_{|\partial \mathcal{D}}&=0,~t\ge 0,\\
x^{\epsilon}(0,\cdot)&=x_0,~y^{\epsilon}(0,\cdot)=y_0,
\end{aligned}
\end{equation*}
where the domain is $\mathcal{D}=(0,1)^d$, for some $d\in\left\{1,2,3\right\}$, and $\dot{W}$ is a Gaussian noise which is white in time, and white or correlated in space. For a mathematically precise formulation, see~\eqref{eq:SPDE}, in the framework of~\cite{DPZ}, and Section~\ref{sec:setting}. The fast component is given by $y^\epsilon(t,\cdot)=y(\epsilon^{-1}t,\cdot)$, and is assumed to be an ergodic process, independent of the Wiener process $W$ which appears in the slow equation, {\it i.e.} the equation for the slow component $x^\epsilon$.

The averaging principle for such SPDE systems has been proved in a very general framework in \cite{Cerrai:2009}, \cite{CerraiFreidlin:2009} for globally Lipschitz coefficients, and later in~\cite{Cerrai:2011} in the case of non-globally Lipschitz continuous coefficients. In these results, no order of convergence, in terms of $\epsilon\to 0$, is provided. The first study of orders of convergence  for the averaging principle in the SPDE case, was performed by the author in~\cite{B:2012}. The main motivation for studying the rates of convergence is the construction of efficient numerical schemes, based on the Heterogeneous Multiscale Methods, see~\cite{B:2013} and references therein. 

In recent years, many works have been devoted to the study of the averaging principle for different classes of SPDEs. For instance, see~\cite{dong2018averaging,FuDuan:11,FuLiu:11,fu2018weak,FuWanWangLiu:14,LiSunXieZhao:18}, in the parabolic SPDE case. See~\cite{FuWanLiu:15,FuWanLiuLiu:18}, for some parabolic-hyperbolic systems. See~\cite{Gao:18,GaoLi:17} in the Schr\"odinger equations case. Finally, see~\cite{BouchetNardiniTangarife:13,BouchetNardiniTangarife:14}, where stochastic fluid mechanics equations are considered, with motivations coming from physics.

As is usual when dealing with stochastic equations, orders of convergence are understood in two senses. First, strong convergence deals with the mean-square error. Second, weak convergence is related to convergence in distribution, considering sufficiently smooth test functions. If the averaging principle for SDEs (with globally Lipschitz continuous coefficients) is considered, the strong order of convergence is $\frac12$, whereas the weak order is $1$, and these results are optimal in general. The technique of~\cite{Khasminskii:1968} is perfectly suited to prove strong convergence results, whereas to study weak convergence, approaches based on asymptotic expansions of solutions of Kolmogorov equations are very efficient, see~\cite{KhasminskiiYin:2004,KhasminskiiYin:2005}. The generalization to SPDEs, where there is no Wiener noise in the slow equation, has been considered in~\cite{B:2012}. If a stochastic forcing term is present, much less is known. Indeed, for SPDEs, {\it i.e.} for infinite dimensional stochastic equations, the analysis of the order of weak convergence and of the Kolmogorov equations, is notoriously challenging, we refer for instance to~\cite{BD} for a recent contribution, and the discussions and references therein.

The aim of this manuscript is to study the weak order of convergence in the averaging principle, for semilinear, parabolic, SPDEs, with a stochastic forcing in the slow equation. Note that this question has been recently investigated in~\cite{fu2018weak} with a very strong regularity condition on the covariance of the noise, which implies a high spatial regularity of the process $x^\epsilon$. In fact, using this condition, the techniques of~\cite{B:2012} may be applied, essentially with no modification, hence weak order of convergence equal to $1$ in~\cite{fu2018weak}. The objective of this manuscript is to obtain similar results with weaker regularity conditions. The main finding is that a trade-off between regularity properties of the slow and the fast components is at play. First, we prove that the strong (resp. weak) order of convergence is $1/2$ (resp. $1$), under an appropriate condition (the {\it very regular case}), which is in general much weaker than the assumption in~\cite{fu2018weak}. Second, we weaken the condition (the {\it less regular case}), and exhibit appropriate strong and weak orders of convergence depending on the regularity properties of the slow and fast component. In that case, as expected, the weak order is twice the strong order, however, whether these results are optimal is not known, indeed the proof is based on an approximation argument and may not be optimal. For statements of the main results, see Theorems~\ref{th:strong_regular},~\ref{th:weak_regular} and~\ref{th:general}.


Even if the main motivation of this work is the analysis of the weak order of convergence, a detailed analysis of the strong order of convergence is also provided, for two reasons. First, it allows us to check that the weak order is twice the strong order, as expected. Second, the technique of proof is different from the one used in previous publications on the strong convergence in the averaging principle for SPDEs, such as~\cite{B:2012} instead of employing the technique introduced by Khasminskii in~\cite{Khasminskii:1968}, the Poisson equation technique described for instance in~\cite{PavliotisStuart} is generalized to a situation where mild solutions of SPDEs are considered.

Several relevant questions are left for future works. For instance, in this manuscript, it is assumed that the fast component $y^\epsilon$ is not coupled with the slow component $x^\epsilon$, and it would be interesting to study the coupled case.


This article is organized as follows. Section~\ref{sec:setting} is devoted the introduction of the functional analysis framework and to stating precise assumptions. Regularity parameters $\aly$ and $\gamy$, which are used to define the very regular and less regular cases, are introduced in Assumptions~\ref{ass:Q} and~\ref{ass:Y} respectively. Section~\ref{sec:averaging} presents the averaged equation.

The main results of this article are stated and discussed in Section~\ref{sec:main}. First, in the very regular case, see Assumption~\ref{ass:regular} and Section~\ref{sec:main_regular}, Theorems~\ref{th:strong_regular} and~\ref{th:weak_regular} state that the strong (resp. weak) order is equal to $\frac12$ (resp. $1$). Second, in the less regular case, see Assumption~\ref{ass:general} and Section~\ref{sec:main_general}, Theorem~\ref{th:general} show that the strong order is (at least) $\beta_{\rm max}=\frac{\aly}{1+\aly-\gamy}\le \frac12$ and the weak order is (at least) $2\beta_{\rm max}$.

Auxiliary but fundamental and nontrivial regularity results concerning a family of Poisson equations are studied in Section~\ref{sec:Poisson}.

Proofs of the main results are provided in Sections~\ref{sec:proof_strong_regular},~\ref{sec:proof_weak_regular} and~\ref{sec:proof_general}.

Finally, in Section~\ref{sec:hmm}, an application of the main result is presented, for the construction and analysis of efficient numerical schemes, based on Heterogeneous Multiscale Methods.

\section{Setting}\label{sec:setting}

The objective of this section is to state precise assumptions, and to derive moment estimates (uniform in $\epsilon$), for the following Stochastic Evolution Equation
\begin{equation}\label{eq:SPDE}
dX^\epsilon(t)=AX^\epsilon(t)dt+F\bigl(X^\epsilon(t),Y^\epsilon(t)\bigr)dt+dW^Q(t).
\end{equation}
This is the abstract formulation of a parabolic, semilinear, SPDE, in the framework of~\cite{DPZ}. The stochastic forcing is given by a $Q$-Wiener process $W^Q$. In addition, $Y^\epsilon$ is another stochastic process with values in $L^2$. In this work, it is assumed that $Y^\epsilon$ and $W^Q$ are independent.

We are interested in the regime of a small parameter $\epsilon$, and of a timescale separation: $Y^\epsilon(t)=Y(t\epsilon^{-1})$. As a consequence, $X^\epsilon$ is referred to as the slow component, and $Y^\epsilon$ as the fast component.

For instance, the process $Y$ may be the solution of an equation of the type
\[
dY(t)=AY(t)dt+G(Y(t))dt+dw^q(t),
\]
where $\bigl(w^q(t)\bigr)_{t\ge 0}$ is a $q$-Wiener process, independent of $W^Q$. Then $Y^\epsilon$ solves an equation of the type
\[
dY^\epsilon(t)=\frac{1}{\epsilon}\bigl(AY^\epsilon(t)+G(Y^\epsilon(t))\bigr)dt+\frac{1}{\sqrt{\epsilon}}dw^q(t),
\]
in which case one has the equality in distribution (but not almost surely) of the processes $\bigl(Y(t\epsilon^{-1})\bigr)_{t\ge 0}$ and $\bigl(Y^\epsilon(t)\bigr)_{t\ge 0}$. The assumption that $Y$ is independent of $W^Q$ means that $G$ does not depend on the slow component, thus the fast evolution is not coupled with the slow evolution.

Considering the coupled situation, where $G$ depends also on the slow component, would substantially modify some computations below. However, the treatment of the uncoupled case, considered in this manuscript, already requires the use of original and nontrivial arguments. The objective of this manuscript is to exhibit these arguments, in the simplest nontrivial framework. The treatment of the coupled case is left for future work.

\subsection{Notation}

Let $\mathcal{D}=(0,1)^d$, with dimension $d\in\left\{1,2,3\right\}$, denote a domain. For any $p\in[2,\infty]$, let $L^p=L^p(\mathcal{D})$, and denote by $|\cdot|_{L^p}$ the associated $L^p$-norm. When $p=2$, $H=L^2$ is a separable, infinite dimensional, Hilbert space, with norm $|\cdot|_H=|\cdot|_{L^2}$, and inner product denoted by $\langle\cdot,\cdot\rangle$.

For any $p,q\in[2,\infty)$, let $\mathcal{L}(L^p,L^q)$ denote the space of bounded linear operators from $L^p$ to $L^q$. The associated norm is denoted by $\|\cdot\|_{\mathcal{L}(L^p,L^q)}$.

For $p\in[2,\infty)$, let $\mathcal{R}(L^2,L^p)\subset \mathcal{L}(L^2,L^p)$ denote the space of $\gamma$-Radonifying  operators from $L^2$ to $L^p$. Recall that a linear operator $\Psi\in\mathcal{L}(L^2,L^p)$ is a $\gamma$-radonifying operator, if the image by $\Psi$ of the canonical gaussian distribution on $L^2$ extends to a Borel probability measure on $L^p$. The space $\mathcal{R}(L^2,L^p)$ is equipped with the norm $\|\cdot\|_{\mathcal{R}(L^2,L^p)}$ defined by
$$
\|\Psi\|_{\mathcal{R}(L^2,L^p)}^2=\tilde{\E}\big| \sum_{n\in\N} \gamma_n \Psi f_n\big|^2,
$$
where $(\gamma_n)_{n\in \N}$ is any sequence of independent standard (mean $0$ and variance $1$) Gaussian random variables, defined on a probability space $(\tilde\Omega,\tilde{\mathcal F}, \tilde{ \mathbb P})$, with expectation operator denoted by $\tilde{\E}$, and $(f_n)_{n\in \N}$ is any complete orthonormal system of $L^2$. When $p=2$, $\mathcal{R}(L^2,L^2)=\mathcal{L}_2(L^2)$ is the space of Hilbert-Schmidt operators on $L^2$, and $\|\Psi\|_{\mathcal{R}(L^2,L^2)}^2={\rm Tr}(\Psi\Psi^\star)$, where ${\rm Tr}(\cdot)$ is the trace operator, and $\Psi^\star$ is the adjoint of $\Psi$.

Note that, for any $p\in[2,\infty)$, there exists $c_p\in(0,\infty)$ such that for any $\Psi\in\mathcal{R}(L^2,L^p)$,
\[
\|\Psi\|_{\mathcal{R}(L^2,L^p)}^2\le c_p\big|\sum_{n\in\N}(\Psi f_n)^2\big|_{L^{\frac{p}{2}}}.
\]

Finally, recall the left and right ideal property for $\gamma$-Radonifying operators: for all $p,q\in[2,\infty)$, for all operators $L_1\in \mathcal{L}(L^p,L^q)$, $\Psi\in \mathcal{R}(L^2,L^p)$ and $L_2\in\mathcal{L}(L^2,L^2)$, then $L_1\Psi L_2\in \mathcal{R}(L^2,L^q)$, and
\[
\|L_1\Psi L_2\|_{\mathcal{R}(L^2,L^q)}\le \|L_1\|_{\mathcal{L}(L^p,L^q)} \|\Psi\|_{\mathcal{R}(L^2,L^p)} \|L_2\|_{\mathcal{L}(L^2,L^2)}.
\]

Let $\bigl(W(t)\bigr)_{t\ge 0}$ denote a cylindrical Wiener process defined on $L^2$, on a probability space $(\Omega,\mathcal{F},\mathbb{P})$. For any $T\in(0,\infty)$ and $p\in[2,\infty)$, the $L^p$-valued It\^o integral $\int_{0}^{T}\Phi(t)dW^Q(t)$ is defined for predictable processes $\Phi\in L^2(\Omega\times (0,T),\mathcal{R}(L^2,L^p))$. Moreover, there exists $c_p\in(0,\infty)$, such that 
\[
\E\bigl[\|\int_0^T \Phi(t)dW(t)\|_{L^p}^2\bigr]\le c_p \int_0^T\E\|\Phi(t)\|_{\mathcal{R}(L^2,L^p)}^2 dt.
\]
In the case $p=2$, the inequality above is replaced by the following version of It\^o isometry property:
\[
\E\bigl[\|\int_0^T \Phi(t)dW(t)\|_{L^2}^2\bigr]=\int_0^T\E\|\Phi(t)\|_{\mathcal{R}(L^2,L^2)}^2 dt.
\]
Higher order moments of stochastic integrals are estimated using Burkholder-Davis-Gundy type inequalities.

For statements, proofs, and generalizations, of the results above, we refer for instance to~\cite{Brzezniak:97,VanNeerven_Veraar_Weis:07,VanNeerven_Veraar_Weis:08} for Banach space valued stochastic integrals, and to~\cite{DPZ} for the Hilbert space case.

If $\varphi:L^2\to\R$ is a function of class $\mathcal{C}^1$, its first order derivative $D\varphi(x)\in\mathcal{L}(L^2,\R)$ may be identified with a element of $L^2$, thanks to Riesz Theorem: as a consequence, for all $x,h\in L^2$, we write $D\varphi(x).h=\langle D\varphi(x),h\rangle$.

\subsection{The linear operator $A$}

Let $A$ denote the unbounded linear operator on $H=L^2$, with
\[
\begin{cases}
D(A)=H^2(0,1)\cap H_0^1(0,1),\\
Ax=\Delta x,~\forall~x\in D(A),
\end{cases}
\]
where $\Delta$ is the Laplace differential operator in dimension $d$. The domain is chosen in order to consider homogeneous Dirichlet boundary conditions in evolution equations. It is a standard result (see for instance~\cite{Pazy}) that there exists a complete orthonormal system $\bigl(e_n\bigr)_{n\in\N}$ of $L^2$, and a non-decreasing sequence $\bigl(\lambda_n\bigr)_{n\in\N}$ of positive real numbers such that
\[
Ae_n=-\lambda_ne_n,~\forall~n\in\N~,\quad \lambda_n\underset{n\to\infty}\sim c_dn^{\frac2d}\quad,\quad \underset{p\in\N}\sup~\underset{n\in\N}\sup~|e_n|_{L^p}<\infty.
\]

The operator $A$ can also be considered as an unbounded linear operator on $L^p$, for all $p\in[2,\infty)$, in a consistent way as $p$ varies. The linear operator $A$ generates an analytic semi-group $\bigl(e^{tA}\bigr)_{t\ge 0}$, on $L^p$ for $p\in[2,\infty)$. For $\alpha\in(0,1)$, the linear operators $(-A)^{-\alpha}$ and $(-A)^{\alpha}$ are constructed in a standard way, see for instance~\cite{Pazy}:
\begin{gather*}
(-A)^{-\alpha}=\frac{\sin(\pi \alpha)}{\pi}\int_{0}^{\infty}t^{-\alpha}(tI-A)^{-1}dt~,~(-A)^{\alpha}=\frac{\sin(\pi \alpha)}{\pi}\int_{0}^{\infty}t^{\alpha-1}(-A)(tI-A)^{-1}dt,
\end{gather*}
where $(-A)^{\alpha}$ is defined as an unbounded linear operator on $L^p$. In the case $p=2$, note that
\begin{gather*}
(-A)^{-\alpha}x=\sum_{i\in \N^{\star}}\lambda_i^{-\alpha} \langle x,e_i\rangle e_i, \quad x\in H,\\
(-A)^{\alpha}x=\sum_{i\in \N^{\star}}\lambda_i^\alpha \langle x,e_i\rangle e_i, \quad x\in D_2\bigl((-A)^{\alpha}\bigr)=\left\{x\in H ; \sum_{i=1}^{\infty}\lambda_{i}^{2\alpha}\langle x,e_i\rangle^2<\infty\right\}.
\end{gather*}

Introduce also the kernel function $K$ associated with the semigroup $\bigl(e^{tA}\bigr)_{t\ge 0}$:
\[
e^{tA}\varphi(\xi)=\int_{\mathcal{D}}K(t,\xi,\eta)\varphi(\eta)d\eta.
\]
This kernel satisfies the two following properties:
\[
0\le K(t,\xi,\eta)\le Ct^{-\frac{d}{2}}\exp\bigl(-ct^{-1}|\xi-\eta|^2\bigr)~,\quad \int_{\mathcal{D}}K(t,\xi,\cdot)\le 1.
\]

To conclude this section, we state useful calculus inequalities, which are employed in a crucial way in Section~\ref{sec:Poisson}.
\begin{propo}\label{propo:inequalities}
For any $\alpha\in[0,\frac12)$, any $\kappa\in(0,\frac12-\alpha)$, and any $p\in[2,\infty)$, there exists $C_{\alpha,\kappa,p}\in(0,\infty)$, such that for all $x_1,x_2$,
\[
|(-A)^\alpha(x_1x_2)|_{L^p}\le C_{\alpha,\kappa,p}|(-A)^{\alpha+\kappa}x_1|_{L^{2p}}|(-A)^{\alpha+\kappa}x_2|_{L^{2p}}.
\]

Moreover, let $\phi:(z_1,z_2)\in \mathbb{R}\times\mathbb{R}\mapsto \phi(z_1,z_2)\in \mathbb{R}$ be a Lipschitz continuous function of class $\mathcal{C}^1$. Then there exists $C(\phi)\in(0,\infty)$ such that for all $x,y_1,y_2$,
\[
|(-A)^{\alpha}\bigl(\phi(x,y_2)-\phi(x,y_1)\bigr)|_{L^p}\le C_{\alpha,\kappa,p}(1+|(-A)^{\alpha+\kappa}x|_{L^{2p}}+\sum_{j=1,2}|(-A)^{\alpha+\kappa}y_j|_{L^{2p}}\bigr)\big|(-A)^{\alpha+\kappa}(y_2-y_1)\big|_{L^{2p}}.
\]
\end{propo}

For the first inequality, we refer to~\cite[Section~3.2]{BD} and~\cite{Triebel}. The second inequality is a straightforward consequence of the first inequality and of a first order Taylor formula:
\[
\phi(x,y_2)-\phi(x,y_2)=\int_{0}^{1}\partial_{z_2}\phi(x,\lambda y_2+(1-\lambda)y_1)(y_2-y_1)d\lambda.
\]

\subsection{Assumptions on $F$ and $Q$}

The coefficient $F$ in~\eqref{eq:SPDE} is defined as the Nemytskii operator (see Definition~\ref{defi:F}) associated with a smooth function $f$ (see Assumption~\ref{ass:f}).

\begin{hyp}\label{ass:f}
Assume that $f:(z_1,z_2)\in \mathbb{R}\times \mathbb{R}\mapsto f(z_1,z_2)\in \mathbb{R}$ is a function of class $\mathcal{C}^{\ir}$, with bounded derivatives of order $1,\ldots,\ir$.
\end{hyp}

\begin{rem}
In the calculations below, quantitative estimates will only depend on the bounds on the derivatives of $f$ of order $1,2,3$. Existence of the fourth order derivative is only employed to justify some calculations in Section~\ref{sec:Poisson}.
\end{rem}

\begin{defi}\label{defi:F}
For all $p,q\in[2,\infty)$, the mapping $F:L^p\times L^p\to L^{p\wedge q}$ is defined as the Nemytskii operator, with $F(x,y)=f\bigl(x(\cdot),y(\cdot)\bigr)$ for all $x\in L^p,y\in L^q$.
\end{defi}

Note that the definition of $F$ is consistent when parameters $p$ and $q$ vary.

Observe that, for any $p,q\in [2,\infty)$, and fixed $y\in L^q$, then the mapping $x\in L^p\mapsto F(x,y)\in L^{p\wedge q}$ is globally Lipschitz continuous, uniformly in $y\in L^q$, and in $p,q$. More precisely,
\[
{\rm Lip}\bigl(F(\cdot,y)\bigr)\le \underset{(z_1,z_2)\in\mathbb{R}^2}\sup~|\partial_{z_1}f(z_1,z_2)|. 
\]

\bigskip

The stochastic perturbation in the slow component of~\eqref{eq:SPDE} is given by a $Q$-Wiener process. The covariance operator $Q$ is a bounded, self-adjoint, operator on $L^2$, and satisfies Assumption~\ref{ass:Q} below.
\begin{hyp}\label{ass:Q}
There exists a family of nonnegative real numbers $\bigl(q_n\bigr)_{n\in\N}$ and a complete orthonormal system $\bigl(f_n\bigr)_{n\in\N}$ of $H$, such that $\underset{n\in\N}\sup~q_n<\infty$, and
\[
Q=\sum_{n\in\N} q_n \langle f_n,\cdot\rangle f_n.
\]
Let $Q^{\frac12}$ be defined as
\[
Q^{\frac12}=\sum_{n\in\N} \sqrt{q_n} \langle f_n,\cdot\rangle f_n.
\]
It is assumed that $f_n\in L^\infty$ for all $n\in \N$, and that
\[
\underset{n\in\N}\sup~|f_n|_{L^\infty}<\infty.
\]
Finally, assume that there exists $\aly\in(0,1]$ such that, for all $\alpha\in[0,\aly)$ and $p\in[2,\infty)$,
\begin{equation}\label{eq:assQ}
M_{\alpha,p}(Q^{\frac12},T)=\bigl(\int_{0}^{T}\|(-A)^{\alpha}e^{tA}Q^{\frac{1}{2}}\|_{\mathcal{R}(L^2,L^p)}^2dt\bigr)^{\frac12}<\infty.
\end{equation}
\end{hyp}

The $Q$-Wiener process $W^Q$ is then defined on a probability space $(\Omega,\mathcal{F},\mathbb{P})$, as follows:
\[
W^Q(t)=\sum_{n\in\N}\sqrt{q_n}\gamma_n f_n,
\]
where $\bigl(\gamma_n\bigr)_{n\in\N}$ is a sequence of independent standard Gaussian random variables. Note that $W^Q(t)=Q^{\frac12}W(t)$ where $W(t)=\sum_{n\in\N}\gamma_n f_n$ is a cylindrical Wiener process.

Note that, for all $T\in(0,\infty)$, $M_{\alpha,p}(Q^\frac12,T)<\infty$ if and only if $M_{\alpha,p}(Q^\frac12,1)<\infty$. Thus the condition expressed in~\eqref{eq:assQ} does not depend on the time $T$, it only depends on $Q$, and on the parameters $\alpha$ and $p$.

A sufficient condition for~\eqref{eq:assQ} to hold is the following: for all $\alpha\in[0,\aly)$ and $p\in[2,\infty)$, $\|(-A)^{\alpha-\frac12}Q^{\frac12}\|_{\mathcal{R}(L^2,L^p)}<\infty$. For another class of sufficient conditions, see Assumption~\ref{ass:general} and Proposition~\ref{propo:suffQ}.

\subsection{Assumptions on the fast process}

The fast process $\bigl(Y^\epsilon(t)\bigr)_{t\ge 0}$ in~\eqref{eq:SPDE} is defined in terms of an ergodic Markov process $Y$, such that, for all $t\ge 0$,
\[
Y^\epsilon(t)=Y(\frac{t}{\epsilon}).
\]

\begin{hyp}\label{ass:Y}
The process $Y=\bigl(Y(t)\bigr)_{t\ge 0}$ is a continuous, ergodic, Markov process on $H=L^2$. Its unique invariant probability distribution is denoted by $\mu$.

Moreover, it is assumed that $Y$ and the $Q$-Wiener process $W^Q$ are independent.

Finally, there exists a parameter $\gamy\in(0,\frac12]$, such that the following estimates are satisfied: for all $\gamma\in[0,\gamy)$, all $p\in[2,\infty)$, and $M\in\mathbb{N}$, there exists $C_{\gamma,p,M}\in(0,\infty)$ such that
\begin{align}
&\underset{t\ge 0}\sup~\E|(-A)^{\gamma}Y(t)|_{L^p}^M\le C_{\gamma,p,M}(1+\E|(-A)^{\gamma}Y(0)|_{L^p}^M),\label{eq:assY_1}\\
&\int |(-A)^\gamma y|_{L^p}^M \mu(dy)\le C_{\gamma,p,M},\label{eq:assY_2}
\end{align}
\end{hyp}

The following standard notation is used: $\bigl(Y_y(t)\bigr)_{t\ge 0}$ denotes the Markov process with initial condition $Y(0)=y$.

Another key assumption concerning the fast process deals with solvability of Poisson equations, and on regularity properties of the solutions.
\begin{hyp}\label{ass:Poisson}
Define admissible functions $\phi:H\to \mathbb{R}$, to be such that, for some $q\in[2,\infty)$, $\phi$ is twice Fr\'echet differentiable on $L^q$, and such that there exists $C\in(0,\infty)$ such that for all $y\in H$, $h,h_1,h_2\in L^q$,
\[
|D\phi(y).h|\le C|h|_{L^q}\quad,~|D^2\phi(y).(h_1,h_2)|\le C|h_1|_{L^q}|h_2|_{L^q}.
\]

Let $\Ly$ be the infinitesimal generator of the Markov process $Y$.

Assume that for any admissible function $\phi$, the Poisson equation
\begin{equation}\label{eq:assPoisson1}
-\Ly \psi=\phi-\int \phi d\mu
\end{equation}
admits a unique solution such that $\int \psi d\mu=0$, and that this solution is given by
\begin{equation}
\psi(y)=\int_0^\infty \E\bigl[\phi(Y_y(t))-\int\phi d\mu\bigr]dt.
\end{equation}

Moreover, for all $\gamma\in[0,\gamy)$, $p\in[2,\infty)$ and $M\in\N_0$, assume that there exists $C_{\gamma,p,M}\in(0,\infty)$ such that the following property is satisfied. Let $\phi:H\to \mathbb{R}$ be an admissible function, and assume that there exists $C(\phi)\in(0,\infty)$, such that for all $y_1,y_2\in H$
\begin{equation}\label{eq:assPoisson2ass}
|\phi(y_2)-\phi(y_1)|\le C(\phi)(1+|(-A)^{\gamma}y_1|_{L^p}^M+|(-A)^{\gamma}y_2|_{L^p}^M)|(-A)^\gamma(y_2-y_1)|_{L^p}.
\end{equation}
Then the solution $\psi$ of the Poisson equation~\eqref{eq:assPoisson1} satisfies, for all $y\in H$,
\begin{equation}\label{eq:assPoisson2}
|\psi(y)|\le C_{\gamma,p,M}C(\phi)(1+|(-A)^\gamma y|_{L^p}^{M+1}).
\end{equation}
\end{hyp}

A sufficient condition to have the estimate~\eqref{eq:assPoisson2} satisfied is given by Proposition~\ref{propo:Y}.
\begin{propo}\label{propo:Y}
Let Assumption~\ref{ass:Y} be satisfied. Assume that for all $\gamma\in[0,\gamy)$, $p\in[2,\infty)$ and $M\in\N_0$, there exists $C_{\gamma,p,M}\in(0,\infty)$ such that for all $y_1,y_2\in L^p$,
\[
\int_{0}^{\infty}\bigl(\E|(-A)^{\gamma}(Y_{y_2}(t)-Y_{y_1}(t))|_{L^p}^{M}\bigr)^{\frac1M}dt\le C_{\gamma,p,M}|(-A)^{\gamma}(y_2-y_1)|_{L^p}.
\]
Then the estimate~\eqref{eq:assPoisson2} is satisfied.
\end{propo}

\begin{proof}[Proof of Proposition~\ref{propo:Y}]
By stationnarity, $\psi$ is written as follows:
\[
\psi(y)=\int \int_{0}^{\infty}\E[\phi(Y^y(t))-\phi(Y^z(t))]dt\mu(dz).
\]
Then, using~\eqref{eq:assPoisson2ass} and Assumption~\ref{ass:Y},
\begin{align*}
|\psi(y)|&\le C(\phi)\int \int_{0}^{\infty}\bigl(\E|(-A)^{\gamma}(Y^y(t)-Y^z(t)|_{L^p}^{2}\bigr)^{\frac12}\bigl(1+(\E|(-A)^\gamma Y^y(t)|_{L^p}^{2M})^{\frac12}+(\E|(-A)^{\gamma}Y^z(t)|_{L^p}^{2M})^{\frac12}\bigr)dt \mu(dz)\\
&\le C_{\gamma,p,M}(\phi)\int \int_{0}^{\infty}\bigl(\E|(-A)^{\gamma}(Y^y(t)-Y^z(t)|_{L^p}^{2}\bigr)^{\frac12}dt\bigl(1+|(-A)^\gamma y|_{L^p}^M+|(-A)^\gamma z|_{L^p}^{M}\bigr)\mu(dz)\\
&\le C_{\gamma,p,M}(\phi)\int |(-A)^\gamma (y-z)|_{L^p}|\bigl(1+|(-A)^\gamma y|_{L^p}^M+|(-A)^\gamma z|_{L^p}^{M}\bigr)\mu(dz)\\
&\le C_{\gamma,p,M}\bigl(1+|(-A)^\gamma y|_{L^p}^{M+1}\bigr).
\end{align*}
\end{proof}

\subsection{Well-posedness and moment estimates}

We are now in position to state (and give a sketch of proof of) a well-posedness result for~\eqref{eq:SPDE}, for arbitrary $\epsilon>0$. Without loss of generality, it is assumed that $\epsilon\in(0,1)$.

We also state moment estimates for $X^\epsilon(t)$. These estimates are uniform with respect to the parameter $\epsilon\in(0,1)$.

\begin{propo}\label{propo:well}
Let $T\in(0,\infty)$. For any $x_0\in H$, any $y_0\in H$, and any $\epsilon\in(0,1)$, the SPDE~\eqref{eq:SPDE} admits a unique mild solution, with initial conditions $X^\epsilon(0)=x_0$, $Y^\epsilon(0)=y_0$, such that for all $t\in[0,T]$,
\begin{equation}\label{eq:SPDE_mild}
X^\epsilon(t)=e^{tA}x_0+\int_{0}^{t}e^{(t-s)A}F(X^\epsilon(s),Y^\epsilon(s))ds+\int_{0}^{t}e^{(t-s)A}dW^Q(s).
\end{equation}
In addition, the following moment estimates are satisfied, uniformly with respect to $\epsilon\in(0,1)$: for any $T\in(0,\infty)$, $\alpha\in[0,\aly)$, $p\ge 2$ and $M\in\N$, there exists $C_{\alpha,p,M}(T)\in(0,\infty)$, such that for all $x_0,y_0\in L^p$, such that $|(-A)^{\alpha}x_0|_{L^p}<\infty$,
\begin{equation}\label{eq:moment}
\underset{\epsilon\in(0,1)}\sup~\underset{t\in[0,T]}\sup~\E\bigl[|(-A)^{\alpha}X^\epsilon(t)|_{L^p}^M\bigr]\le C_{\alpha,p,M}(T)\bigl(1+|(-A)^{\alpha}x_0|_{L^p}^M+|y_0|_{L^p}^{M}+M_{\alpha,p}(Q^\frac12,T)^M\bigr).
\end{equation}
\end{propo}

\begin{rem}
Using regularization properties of the semigroup $\bigl(e^{tA}\bigr)_{t\ge 0}$, the moment estimates~\eqref{eq:moment} may be replaced with
\[
\underset{\epsilon\in(0,1)}\sup~\E\bigl[|(-A)^{\alpha}X^\epsilon(t)|_{L^p}^M\bigr]\le C_{\alpha,p,M}(T)\bigl(1+t^{-\alpha M}|x_0|_{L^p}^M+|y_0|_{L^p}^{M}+M_{\alpha,p}(Q^\frac12,T)^M\bigr).
\]
Therefore the regularity assumption on the initial condition $x_0$ may be relaxed.
\end{rem}

We conclude this section with a sketch of proof of Proposition~\ref{propo:well}.
\begin{proof}[Proof of Proposition~\ref{propo:well}]
The existence and uniqueness of a mild solution~\eqref{eq:SPDE_mild} of~\eqref{eq:SPDE} is obtained by a standard fixed point argument, see for instance~\cite{DPZ}.

The proof of the moment estimates~\eqref{eq:moment} combines the following observations. On the one hand, by regularization properties of the semigroup and Lipschitz continuity of $F$,
\[
\big|(-A)^\alpha\int_{0}^{t}e^{(t-s)A}F(X^\epsilon(s),Y^\epsilon(s))ds\big|_{L^p}\le C\int_{0}^{t}(t-s)^{-\alpha}\bigl(1+|X^\epsilon(s)|_{L^p}+|Y^\epsilon(s)|_{L^p}\bigr)ds.
\]
On the other hand, thanks to~\eqref{eq:assQ}, see Assumption~\ref{ass:Q}, the moment estimate
\[
\E|(-A)^\alpha\int_{0}^{t}e^{(t-s)A}dW^Q(s)|_{L^p}^2\le c_p\int_{0}^{T}\|(-A)^{\alpha}e^{tA}Q^{\frac{1}{2}}\|_{\mathcal{R}(L^2,L^p)}^2dt=c_pM_{\alpha,p}(Q^\frac12,T)^2<\infty
\]
for the stochastic convolution, is easily obtained, in the case $M=2$. Higher order moments are estimated using a Burkholder-Davis-Gundy type inequality.

The case $\alpha=0$ is treated using Gronwall inequality, and then the case $\alpha\in(0,\aly)$ follows from the estimates above.

This concludes the sketch of proof of Proposition~\ref{propo:well}.
\end{proof}

\section{The averaging principle}\label{sec:averaging}

Let us first define the so-called averaged coefficient $\Fb$.
\begin{defi}
For any $x\in L^2$, define
\begin{equation}
\Fb(x)=\int F(x,y) d\mu(y)\in L^2,
\end{equation}
where $\mu$ is the unique invariant probability distribution of the ergodic process $Y$, see Assumption~\ref{ass:Y}.
\end{defi}

Since $F$ is the Nemytskii operator associated with a globally Lipschitz continuous function $f$, using Assumption~\ref{ass:Y}, it is straightforward to check that $\Fb(x)\in L^p$ if $x\in L^p$, for all $p\in[2,\infty)$.

The first and second order derivatives of the averaged coefficient $\Fb$ satisfy the following estimates.
\begin{propo}\label{propoFb}
For all $p\in[2,\infty)$, there exists $C_p\in(0,\infty)$ such that for all $h\in L^p$
\begin{equation}\label{eq:propoFb1}
\underset{x\in L^2}\sup~|D\Fb(x).h|_{L^p}\le C_p|h|_{L^p}.
\end{equation}

For all $p\in[2,\infty)$, and $p_1,p_2\in[2,\infty)$, such that $\frac{1}{p}=\frac{1}{p_1}+\frac{1}{p_2}$, there exists $C_{p_1,p_2}\in(0,\infty)$ such that for all $h_1\in L^{p_1}$ and $h_2\in L^{p_2}$,
\begin{equation}\label{eq:propoFb2}
\underset{x\in L^2}\sup~|D^2\Fb(x).(h_1,h_2)|_{L^p}\le C_{p_1,p_2}|h_1|_{L^{p_1}}|h_2|_{L^{p_2}}.
\end{equation}
\end{propo}

\begin{proof}[Proof of Proposition~\ref{propoFb}]
Note that for all $x,y\in L^2$ and $h\in L^p$, $h_1\in L^{p_1}$ and $h_2\in L^{p_2}$,
\begin{align*}
& D_xF(x,y).h=\partial_{z_1}f(x,y)h\\
& D_xxF(x,y).(h_1,h_2)=\partial_{z_1}^2f(x,y)h_1 h_2.
\end{align*}
The conclusion follows using boundedness of the first and second order derivatives of $f$, H\"older inequality, and integrating with respect to $\mu(y)$.
\end{proof}

We are now in position to state a well-posedness result for the averaged equation:
\begin{equation}\label{eq:av}
d\Xb(t)=A\Xb(t)dt+\Fb(\Xb(t))dt+dW^Q(t).
\end{equation}

\begin{propo}\label{propo:well_av}
Let $T\in(0,\infty)$. For any $x_0\in H$, the SPDE~\eqref{eq:av} admits a unique mild solution, with initial condition $\Xb(0)=x_0$, such that for all $t\in[0,T]$,
\begin{equation}\label{eq:av_mild}
\Xb(t)=e^{tA}x_0+\int_{0}^{t}e^{(t-s)A}\Fb(\Xb(s))ds+\int_{0}^{t}e^{(t-s)A}dW^Q(s).
\end{equation}
In addition, the following moment estimates are satisfied, uniformly with respect to $\epsilon\in(0,1)$: for any $T\in(0,\infty)$, $\alpha\in[0,\aly)$, $p\ge 2$ and $M\in\N$, there exists $C_{\alpha,p,M}(T)\in(0,\infty)$, such that for all $x_0,y_0\in L^p$, such that $|(-A)^{\alpha}x_0|_{L^p}<\infty$,
\begin{equation}\label{eq:moment_av}
\underset{t\in[0,T]}\sup~\E\bigl[|(-A)^{\alpha}\Xb(t)|_{L^p}^M\bigr]\le C_{\alpha,p,M}(T)\bigl(1+|(-A)^{\alpha}x_0|_{L^p}^M+M_{\alpha,p}(Q^\frac12,T)^M\bigr).
\end{equation}
\end{propo}

The proof is omitted. Existence and uniqueness of the mild solution follows from the global Lipschitz continuity property of $\Fb$ (thanks to~\eqref{eq:propoFb1}, see Proposition~\ref{propoFb}). The moment estimates~\eqref{eq:moment_av} are proved using the same arguments as in the proof of Proposition~\ref{propo:well}, in particular using~\eqref{eq:assQ}, see Proposition~\ref{ass:Q}.

To conclude this section, the infinitesimal generator associated with the averaged equation~\eqref{eq:av} is introduced:
\begin{equation}\label{eq:gen_av}
\Lb\varphi(x)=\langle D_x\varphi(x),Ax+\Fb(x)\rangle+\frac12 \sum_{n\in\N}q_nD_x^2\varphi(x).\bigl(f_n,f_n\bigr).
\end{equation}
This definition makes sense for sufficiently regular functions $\varphi:L^2\to\R$.

\section{Statements of the main results}\label{sec:main}

This section is devoted to the statements of the main results of this article, concerning the error in the averaging principle. We exhibit both strong and weak orders of convergence, with respect to $\epsilon$. Two situations need to be considered, depending on the regularity properties of the slow and the fast component, more precisely in terms of $\gamy$ (see Assumption~\ref{ass:Y}) and $\aly$ (see Assumption~\ref{ass:Q}).

Let us introduce Assumption~\ref{ass:regular} (resp. Assumption~\ref{ass:general}) which defines the {\it regular} case (resp. the {\it general} case).
\begin{hyp}\label{ass:regular}
The parameters $\aly$ and $\gamy$ satisfy the condition
\[
\aly+\gamy>1.
\]
Moreover, assume that ${\rm Tr}(Q)=\sum_{n\in\N}q_n<\infty$.
\end{hyp}

\begin{hyp}\label{ass:general}
Assume that $\bigl(\sum_{n\in\N}q_n^{\frac{\varrho}{2}}\bigr)^{\frac{2}{\rho}}<\infty$, for some
\[
\varrho\in\begin{cases}[2,\infty],~d=1,\\
[2,\frac{d}{d-2}),~d=2,\end{cases}
\]
with usual conventions if $\varrho=\infty$ or if $d=2$.

Let $\aly=\frac{1}{2}\bigl(1-\frac{d}{2}(1-\frac{2}{\varrho})\bigr)$. Moreover, assume that $\gamy\le\aly$.
\end{hyp}

The definition of the parameter $\aly$ in Assumption~\ref{ass:general} is consistent with Assumption~\ref{ass:Q}, see Proposition~\ref{propo:suffQ}. Note that $\aly\in(0,\frac12]$.

\begin{rem}
If Assumption~\ref{ass:general} is satisfied, the condition $\aly\ge \gamy$ is not restrictive. Indeed, in practice, when the regularity $\aly$ is given, one may always replace $\gamy$ with $\min(\aly,\gamy)$ without extra assumption.
\end{rem}

\begin{rem}
The condition~${\rm Tr}(Q)<\infty$ in Assumption~\ref{ass:regular} is not very restrictive. For instance, if $\aly$ is caracterized by the property that for all $\alpha<\aly$ and all $p\in[2,\infty)$, $\|(-A)^{\alpha-\frac12}Q^{\frac12}\|_{\mathcal{R}(L^2,L^p)}<\infty$, the condition is satisfied since ${\rm Tr}(Q)=\|Q^{\frac12}\|_{\mathcal{R}(L^2,L^2)}<\infty$, with $\alpha=\frac12<\aly$.
\end{rem}



First, in the very regular case (see Section~\ref{sec:main_regular}), the strong (resp. weak) order of convergence is equal to $\frac12$ (resp. $1$). This coincides with the orders of convergence obtained in~\cite{B:2012}, where no stochastic perturbation is acting in the slow component, {\it i.e.} $Q=0$ in~\eqref{eq:SPDE}. The weak order $1$ also essentially coincides with the result from~\cite{fu2018weak}, where it is assumed that $\aly=1$ and $\gamy=0$. Moreover, this also coincides with the orders obtained in the case of SDEs (see for instance~\cite{PavliotisStuart}). In particular, these values are optimal in general.

Second, in the less regular case (see Section~\ref{sec:main_general}), Assumption~\ref{ass:general} is satisfied, and it is proved that the strong (resp. weak) order of convergence is equal to $\frac{\aly}{1+\aly-\gamy}$ (resp. $\frac{2\aly}{1+\aly-\gamy}$).  The proof is based on the application of the result in the regular case for a well-chosen approximate problem, with modified covariance operator $Q$. It is not known whether these strong and weak orders of convergence are optimal. On the one hand, observe the orders of convergence are maximal when $\gamy=\aly$, in which case the strong and weak orders are $\aly$ and $2\aly$ respectively, hence are clearly related to the spatial and temporal regularity of the processes. On the other hand, when $\gamy$ is arbitrarily small, the strong and weak orders of convergence are $\frac{\aly}{1+\aly}$ and $\frac{2\aly}{1+\aly}$ respectively. The application of the standard Khasminskii strategy would also lead to a strong order of convergence equal to $\frac{\aly}{1+\aly}$, see~\cite{B:2012}. As a consequence, the additional use of regularity properties of the fast process in the analysis allows us to get improved orders of convergence.



\subsection{The very regular case}\label{sec:main_regular}

In this section, it is assumed that Assumption~\ref{ass:regular} is satisfied. As a consequence, there exists $\gamma\in(1-\aly,\gamy)$, such that $M_{1-\gamma,p}(Q^\frac12,T)<\infty$ for all $p\in[2,\infty)$ and all $T\in(0,\infty)$, see~\eqref{eq:assQ}. In particular,
\[
M_{1-\gamma,8}(Q^\frac12,T)<\infty,
\]
for all $T\in(0,\infty)$,.

We are now in position to provide precise statements of the results, concerning the order of convergence of the averaging error, in the regular case.

\begin{theo}\label{th:strong_regular}
Let Assumption~\ref{ass:regular} be satisfied. Let $T\in(0,\infty)$, and assume that the initial conditions $x_0$, $y_0$, satisfy
\[
|(-A)^{1-\gamma}x_0|_{L^8}+|(-A)^{\gamma+\kappa}y_0|_{L^8}<\infty,
\]
for some $\gamma\in(1-\aly,\gamy)$ and some $\kappa\in(0,\gamy-\gamma)$.

Then there exists $C(T,x_0,y_0)\in(0,\infty)$ such that for all $\epsilon\in(0,1)$,
\begin{equation}\label{eq:th_strong_regular}
\underset{t\in[0,T]}\sup~\bigl(\E|X^\epsilon(t)-\Xb(t)|_{L^2}^2\bigr)^{\frac12}\le C(T,x_0,y_0)\bigl({\rm Tr}(Q)+M_{1-\gamma,8}(Q^\frac12,T)\bigr)^{\frac12}\epsilon^{\frac12}.
\end{equation}
\end{theo}

To state the weak error result, an appropriate notion of admissible test function is used.
\begin{defi}
Let $\varphi:L^2\to\R$. It is called admissible if the following derivatives of $\varphi$ exist and are continuous, and if the estimates below are satisfied.
\begin{itemize}
\item There exists $C\in(0,\infty)$ such that for all $x\in L^2$ and $h\in L^2$,
\[
|D\varphi(x).h|\le C|h|_{L^2},
\]
and, for all $x\in L^2$, $h_1,h_2\in L^2$,
\[
|D^2\varphi(x).(h_1,h_2)|\le C|h_1|_{L^2}|h_2|_{L^2}.
\]
\item For every $p_1,p_2,p_3\in[2,\infty)$ such that $1=\frac{1}{p_1}+\frac{1}{p_2}+\frac{1}{p_3}$, there exists $C_{p_1,p_2,p_3}\in(0,\infty)$ such that for all $x\in L^2$, and $h_1\in L^{p_1},h_2\in L^{p_2},h_3\in L^{p_3}$,
\[
|D^3\varphi(x).(h_1,h_2,h_3)|\le C_{p_1,p_2,p_3}|h_1|_{L^{p_1}}|h_2|_{L^{p_2}}|h_3|_{L^{p_3}}.
\]
\end{itemize}
\end{defi}

For instance, a function $\varphi:L^2\to \R$ of class $\mathcal{C}^3$, with bounded derivatives of order $1,2,3$, is admissible. Other admissible functions are constructed as follows:
\[
\varphi(x)=\langle \omega,\tilde{\varphi}(x)\rangle,
\]
where $\tilde{\varphi}:\R\to\R$ is of class $\mathcal{C}_b^3$, and $\omega\in L^\infty$ is a weight function.

The weak error is estimated for the class of admissible test functions introduced above.
\begin{theo}\label{th:weak_regular}
Let Assumption~\ref{ass:regular} be satisfied. Let $T\in(0,\infty)$, and assume that the initial conditions $x_0$, $y_0$, satisfy
\[
|(-A)^{1-\gamma}x_0|_{L^8}+|(-A)^{\gamma+\kappa}y_0|_{L^8}<\infty,
\]
for some $\gamma\in(1-\aly,\gamy)$ and some $\kappa\in(0,\gamy-\gamma)$. Let $\varphi:L^2\to \R$ be an admissible test function.

There exists $C(T,x_0,y_0,\varphi)\in(0,\infty)$ such that for all $\epsilon\in(0,1)$,
\begin{equation}\label{eq:th_weak_regular}
\underset{t\in[0,T]}\sup~|\E[\varphi(X^\epsilon(t))]-\E[\varphi(\Xb(t))]|\le C(T,x_0,y_0,\varphi)\bigl({\rm Tr}(Q)+M_{1-\gamma,4}(Q^\frac12,T)\bigr)\epsilon.
\end{equation}
\end{theo}

The apparently strong conditions imposed on the initial conditions $x_0$ and $y_0$ may be weakened using standard arguments, thanks to the regularization properties of the semigroup $\bigl(e^{tA}\bigr)_{t\ge 0}$, and minor modifications in the proofs. However, one could not consider the supremum over time $t\in[0,T]$ in~\eqref{eq:th_strong_regular} and~\eqref{eq:th_weak_regular}. In addition, assuming that the initial conditions possess nice spatial regularity properties allows us to focus on the most important issues solved in this manuscript.

Note that the strong order of convergence is equal to $\frac12$, whereas the weak order of convergence is equal to $1$. As explained above, these values are optimal.

The proofs of Theorems~\ref{th:strong_regular} and~\ref{th:weak_regular} are postponed to Sections~\ref{sec:proof_strong_regular} and~\ref{sec:proof_weak_regular} respectively.

\subsection{The less regular case}\label{sec:main_general}

In this section it is assumed that Assumption~\ref{ass:general} is satisfied. Let
\[
\beta_{\rm max}=\frac{\aly}{1+\aly-\gamy},
\]
and observe that $\beta_{\rm max}\le \frac12$.

\begin{theo}\label{th:general}
Let Assumption~\ref{ass:general} be satisfied. Let $T\in(0,\infty)$, and assume that the initial conditions $x_0$, $y_0$, satisfy
\[
|(-A)^{1-\gamma}x_0|_{L^8}+|(-A)^{\gamma+\kappa}y_0|_{L^8}<\infty,
\]
for some $\gamma\in(0,\gamy)$ and some $\kappa\in(0,\aly-\gamma)$.

For any $\beta\in[0,\beta_{\rm max})$, there exists $C_\beta(T,x_0,y_0,Q)\in(0,\infty)$ such that for all $\epsilon\in(0,1)$, the strong error is estimated by
\begin{equation}\label{eq:th_strong_general}
\underset{t\in[0,T]}\sup~\bigl(\E|X^\epsilon(t)-\Xb(t)|_{L^2}^2\bigr)^{\frac12}\le C_\beta(T,x_0,y_0,Q)\epsilon^{\beta}.
\end{equation}

Moreover, let $\varphi:L^2\to \R$ be an admissible test function. For any $\beta\in[0,\beta_{\rm max})$, there exists $C_\beta(T,x_0,y_0,Q,\varphi)\in(0,\infty)$ such that for all $\epsilon\in(0,1)$, the weak error is estimated by
\begin{equation}\label{eq:th_weak_general}
\underset{t\in[0,T]}\sup~|\E[\varphi(X^\epsilon(t))]-\E[\varphi(\Xb(t))]|\le C_\beta(T,x_0,y_0,Q,\varphi)\epsilon^{2\beta}.
\end{equation}
\end{theo}
Note that in Theorem~\ref{th:general}, the weak order is equal to twice the strong order, as discussed above. The proof of Theorem~\ref{th:general} is postponed to Section~\ref{sec:proof_general}.

\section{Auxiliary regularity results for solutions of the Poisson equation}\label{sec:Poisson}

This section is devoted to the analysis of the Poisson equation below: for any $x\in L^2$ and $\theta\in L^2$, define $\Phi(x,\cdot,\theta):L^2\to\R$ as the unique solution of
\begin{equation}\label{eq:Poisson}
-\Ly \Phi(x,\cdot,\theta)=\langle F(x,\cdot)-\overline{F}(x),\theta\rangle,
\end{equation}
with the condition $\int \Phi(x,\cdot,\theta) d\mu=0$. Observe that $\theta\mapsto \Phi(x,y,\theta)$ is a (possibly unbounded ) linear mapping.

Recall that $\Ly$ is the generator of the Markov process $Y$. It is assumed that Assumption~\ref{ass:Poisson} is satisfied.

The function $\Phi$ plays a key role in the analysis of the error in the averaging principle, both in the strong and in the weak senses. It is straightforward to obtain estimates on $\Phi(x,y,\theta)$, on $D_x\Phi(x,y,\theta).h$ and on $D_x^2\Phi(x,y,\theta).(h_1,h_2)$, in terms on $L^p$ norms of $x,y,\theta,h,h_1,h_2$ (for well-chosen $p$), see Lemmas~\ref{lem:Phi_0},~\ref{lem:Phi_1} and~\ref{lem:Phi_2} below. The main original results in this manuscript are estimates of $\Phi(x,y,\theta)$ in terms of $|(-A)^{-\gamma}\theta|_{L^p}$ (see Lemma~\ref{lem:Phi_0ter}), and of $D_x\Phi(x,y,\theta).h$ in terms of $|(-A)^{-\gamma}h|_{L^p}$ (see Lemma~\ref{lem:Phi_1bis}), for positive $\gamma\in(0,\gamy)$. These two results are specific to the analysis of the averaging principle for parabolic SPDEs, and they allow us to exhibit the trade-off between the regularity properties of the slow and fast processes in the identification of the strong and weak orders of convergence discussed above. These results are consequences of Proposition~\ref{propo:inequalities}.

First, Lemma~\ref{lem:Phi_0} and~\ref{lem:Phi_0ter} deal with estimates of $\Phi(x,y,\theta)$. In particular, note that Lemma~\ref{lem:Phi_0} implies the well-posedness of~\eqref{eq:Poisson}.
\begin{lemma}\label{lem:Phi_0}
Let $p\in[2,\infty)$ and $p'=\frac{p}{p-1}\in(1,2]$. There exists $C_p\in(0,\infty)$, such that for all $x\in L^2$, $y\in L^{p}$ and all $\theta\in L^{p'}$,
\[
|\Phi(x,y,\theta)|\le C_q(1+|y|_{L^{p}})|\theta|_{L^{p'}}.
\]
\end{lemma}

\begin{proof}
For any fixed $x\in L^2$ and $\theta\in L^2\subset L^{p'}$, the mapping $y\mapsto \langle F(x,y)-\overline{F}(x),\theta\rangle$ is an admissible function (with $q=4$). In addition, using Lipschitz continuity of $F$, one has the estimate
\[
|\langle F(x,y_2)-F(x,y_1),\theta\rangle|\le |F(x,y_2)-F(x,y_1)|_{L^{p}}|\theta|_{L^{p'}}\le C|y_2-y_1|_{L^{p}}|\theta|_{L^{p'}}.
\]
This proves that~\eqref{eq:assPoisson2ass} is satisfied, with the parameters $\alpha=0$, $p$, and $M=0$. By Assumption~\ref{ass:Poisson}, then~\eqref{eq:assPoisson2} is satisfied, which concludes the proof of Lemma~\ref{lem:Phi_0}.
\end{proof}

\begin{lemma}\label{lem:Phi_0ter}
Let $\gamma\in(0,\gamy)$. For all $\kappa\in(0,\gamy-\gamma)$, there exists $C_{\gamma,\kappa}\in(0,\infty)$ such that for all $x,y\in L^4$ and $\theta\in L^2$, then
\[
|\Phi(x,y,\theta)|\le C_{\gamma,\kappa}\bigl(1+|(-A)^{\gamma+\kappa}x|_{L^4}^2+|(-A)^{\gamma+\kappa}y|_{L^4}^2\bigr)|(-A)^{-\gamma}\theta|_{L^2}.
\]
\end{lemma}

\begin{proof}
Observe that
\begin{align*}
|\langle F(x,y_2)-F(x,y_2),\theta\rangle|&\le C_\gamma|(-A)^{\gamma}(F(x,y_2)-F(x,y_1))|_{L^2}|(-A)^{-\gamma}\theta|_{L^2}\\
&\le C_\gamma\bigl(1+|(-A)^{\gamma+\kappa}x|_{L^4}+|(-A)^{\gamma+\kappa}y_1|_{L^4}+|(-A)^{\gamma+\kappa}y_2|_{L^4}\bigr)\\
&\hspace{4cm}|(-A)^{\gamma+\kappa}(y_2-y_1)|_{L^4}|(-A)^{-\gamma}\theta|_{L^2},
\end{align*}
using the third inequality in Proposition~\ref{propo:inequalities}. This proves that~\eqref{eq:assPoisson2ass} is satisfied, thus~\eqref{eq:assPoisson2} follows, and this concludes the proof of Lemma~\ref{lem:Phi_0ter}.
\end{proof}

Lemmas~\ref{lem:Phi_1} and~\ref{lem:Phi_1bis} deal with the first order derivative of $\Phi(x,y,\theta)$ with respect to $x$.

\begin{lemma}\label{lem:Phi_1}
There exists $C\in(0,\infty)$, such that for all $x\in L^2$, $y,\theta,h\in L^4$,
\[
|\langle D_x\Phi(x,y,\theta),h\rangle|\le C(1+|y|_{L^4})\min\Bigl(|\theta|_{L^4}|h|_{L^2},|\theta|_{L^2}|h|_{L^4}\Bigr).
\]
Moreover, for all $x\in L^2$, $y,h\in L^8$, $\theta\in L^{\frac43}$, one has
\[
|\langle D_x\Phi(x,y,\theta),h\rangle|\le C(1+|y|_{L^8})|\theta|_{L^{\frac43}}|h|_{L^8}.
\]
\end{lemma}

\begin{proof}
For all $x,h\in L^2$, $\theta\in L^4$, the function $y\mapsto \langle D_x\Phi(x,y,\theta),h\rangle$ solves the Poisson equation
\[
-\Ly\bigl(D_x\Phi(x,\cdot,\theta).h\bigr)=\phi_{x,\theta,h}
\]
where $\phi_{x,\theta,h}(y)=\langle D_x\bigl(F(x,\cdot)-\Fb(x)\bigr).h,\theta\rangle$. It is straightforward to check that $\phi_{x,\theta,h}$ is an admissible function (by Assumption~\ref{ass:f}, $f$ is of class $\mathcal{C}^3$ with bounded derivatives), with $q=8$.

Let $x,h\in L^2$ and $\theta\in L^4$, then for all $y_1,y_2\in L^4$, one has
\begin{align*}
|\phi_{x,\theta,h}(y_2)-&\phi_{x,\theta,h}(y_1)|=\big|\langle D_x\bigl(F(x,y_2)-F(x,y_1)\bigr).h,\theta\rangle\big|\\
&=\big|\langle \bigl(\partial_{z_1}f(x,y_2)-\partial_{z_1} f(x,y_1)\bigr)h,\theta\rangle\big|\\
&=\big|\langle \bigl(\partial_{z_1}f(x,y_2)-\partial_{z_1} f(x,y_1)\bigr)\theta,h\rangle\big|\\
&\le C|y_2-y_1|_{L^4}\min\Bigl(|\theta|_{L^4}|h|_{L^2},|\theta|_{L^2}|h|_{L^4}\Bigr),
\end{align*}
using H\"older inequality and by boundedness of the first-order partial derivative $\partial_{z_1}f(z_1,z_2)$.

Alternatively,
\[
|\phi_{x,\theta,h}(y_2)-\phi_{x,\theta,h}(y_1)|\le C|y_2-y_1|_{L^8}|\theta_{L^{\frac43}}|h|_{L^8}.
\]
Thus~\eqref{eq:assPoisson2ass}, and consequently~\eqref{eq:assPoisson2}, are satisfied, for an appropriate choice of the parameters. This concludes the proof of Lemma~\ref{lem:Phi_1}.
\end{proof}

\begin{lemma}\label{lem:Phi_1bis}
Let $\gamma\in(0,\gamy)$. For all $\kappa\in(0,\gamy-\gamma)$, there exists $C_{\gamma,\kappa}\in(0,\infty)$ such that for all $x,y,\theta\in H$, then
\begin{align*}
|\langle D_x\Phi(x,y,\theta),h\rangle|\le C_{\gamma,\kappa}&\bigl(1+|(-A)^{\gamma+\kappa}x|_{L^8}^2+|(-A)^{\gamma+\kappa}y|_{L^8}^2\bigr)\\
&\min\Bigl(|(-A)^{\gamma+\frac{\kappa}{2}}\theta|_{L^4}|(-A)^{-\gamma}h|_{L^2},|(-A)^{\gamma+\frac{\kappa}{2}}\theta|_{L^2}|(-A)^{-\gamma}h|_{L^4}\Bigr).
\end{align*}
\end{lemma}

\begin{proof}
Let $x,\theta,h$ be fixed. Proceeding as in the proof of Lemma~\ref{lem:Phi_1}, for all $y_1,y_2\in H$,
\begin{align*}
|\phi_{x,\theta,h}(y_2)-&\phi_{x,\theta,h}(y_1)|=\big|\langle \bigl(\partial_xf(x,y_2)-\partial_x f(x,y_1)\bigr)h,\theta\rangle\big|=\big|\langle \bigl(\partial_xf(x,y_2)-\partial_x f(x,y_1)\bigr)\theta,h\rangle\big|,
\end{align*}
thus, thanks to H\"older inequality and to the first inequality in Proposition~\ref{propo:inequalities}, it is sufficient to consider
\begin{align*}
|(-A)^{\gamma+\frac{\kappa}{2}}\bigl(\partial_xf(x,y_2)-\partial_x f(x,y_1)\bigr)|_{L^4}\le C_{\gamma,\kappa}(1+|(-A)^{\gamma+\kappa}y_1|_{L^8}+|(-A)^{\gamma+\kappa}y_2|_{L^8})|(-A)^{\gamma+\kappa}(y_2-y_1)|_{L^8},
\end{align*}
thanks to the second inequality of Proposition~\ref{propo:inequalities}.
It remains to use Assumption~\ref{ass:Poisson} to conclude the proof of Lemma~\ref{lem:Phi_1}.
\end{proof}

Finally, it remains to state and prove a result, Lemma~\ref{lem:Phi_2}, concerning the second order derivative.
\begin{lemma}\label{lem:Phi_2}
There exists $C\in(0,\infty)$ such that for all $x\in L^2$ $y,\theta\in L^4$ and $h_1,h_2\in L^8$,
\[
|D_{x}^2\Phi(x,y,\theta).(h_1,h_2)|\le C(1+|y|_{L^4})\min\Bigl(|\theta|_{L^4}|h_1|_{L^4}|h_2|_{L^4},|\theta|_{L^2}|h_1|_{L^8}|h_2|_{L^8}\Bigr).
\]
\end{lemma}

\begin{proof}
For all $x$, $\theta$, $h_1,h_2$, the function $y\mapsto D_x^2\Phi(x,y,\theta).(h_1,h_2)$ solves the Poisson equation
\[
-\Ly\bigl(D_x^2\Phi(x,\cdot,\theta).(h_1,h_2)\bigr)=\phi_{x,\theta,h_1,h_2}^{(2)}
\]
where $\phi_{x,\theta,h_1,h_2}^{(2)}(y)=\langle D_x^2\bigl(F(x,\cdot)-\Fb(x)\bigr).(h_1,h_2),\theta\rangle$. It is straightforward to check that $\phi_{x,\theta,h}^{(2)}$ is an admissible function (thanks to Assumption~\ref{ass:f}, $f$ is of class $\mathcal{C}^4$ with bounded derivatives of order $1,\ldots,4$).

For all $y_1,y_2\in H$, using boundedness of the third-order derivative $\partial_x^{(3)}f$ and H\"older inequality, one obtains
\begin{align*}
|\phi_{x,\theta,h_1,h_2}^{(2)}(y_2)-\phi_{x,\theta,h_1,h_2}^{(2)}(y_1)|&\le C|y_2-y_1|_{L^4}\min\Bigl(|\theta|_{L^4}|h_1|_{L^4}|h_2|_{L^4},|\theta|_{L^2}|h_1|_{L^8}|h_2|_{L^8}\Bigr).
\end{align*}
Thus it remains to apply Assumtion~\ref{ass:Poisson} to conclude the proof of Lemma~\ref{lem:Phi_2}.
\end{proof}

\section{Proof of Theorem~\ref{th:strong_regular}}\label{sec:proof_strong_regular}

The goal of this section is to provide a proof of Theorem~\ref{th:strong_regular}, {\it i.e.} that under Assumption~\ref{ass:regular}, the strong order of convergence in the averaging principle is equal to $\frac12$.

Let $T\in(0,\infty)$. Thanks to Assumption~\ref{ass:regular}, let also $\gamma\in(1-\aly,\gamy)$, $\kappa\in(0,\gamy-\gamma)$, and let the initial conditions $x_0$ and $y_0$ satisfy $|(-A)^{1-\gamma}x_0|_{L^8}+|(-A)^{\gamma+\kappa}y_0|_{L^8}<\infty$.

Introduce the auxiliary function $\delta F(x,y)=F(x,y)-\overline{F}(x)$. Thanks to the mild formulations~\eqref{eq:SPDE_mild} and~\eqref{eq:av_mild}, the following decomposition of the averaging error is obtained:
\begin{align*}
X^\epsilon(t)-\Xb(t)&=\int_{0}^{t}e^{(t-s)A}\bigl(F(X^\epsilon(s),Y^\epsilon(s))-F(\Xb(s),Y^\epsilon(s))\bigr)ds\\
&~+\int_{0}^{t}e^{(t-s)A}\delta F(\Xb(s),Y^\epsilon(s))ds.
\end{align*}
Recall that $F$ is globally Lipschitz-continuous, thanks to Assumption~\ref{ass:f}. The mean-square error is then bounded from above as follows:
\begin{align*}
\E\big|X^\epsilon(t)-\Xb(t)|^2&\le CT\int_{0}^{t}\E\big|X^\epsilon(s)-\Xb(s)\big|^2ds\\
&~+2\int_{0}^{t}\int_{s}^{t}\E\Bigl[\langle e^{(t-s)A}\delta F(\Xb(s),Y^\epsilon(s)),e^{(t-r)A}\delta F(\Xb(r),Y^\epsilon(r))\rangle\Bigr] dr ds.
\end{align*}

Let $\theta_{s,t}(r)=e^{(2t-s-r)A}\delta F(\Xb(s),Y^\epsilon(s))$. Observe that $\partial_r \theta_{s,t}(r)=-A\theta_{s,t}(r)$. Using the definition~\eqref{eq:Poisson} of $\Phi$, considering the quantity~$\E\Phi(\Xb(t),Y^\epsilon(t),\theta_{s,t}(t))-\E\Phi(\Xb(s),Y^\epsilon(s),\theta_{s,t}(s))$, and applying It\^o formula, one obtains
\begin{align*}
\int_{s}^{t}\E\Bigl[\langle e^{(t-s)A}\delta F(\Xb(s),Y^\epsilon(s)),&e^{(t-r)A}\delta F(\Xb(r),Y^\epsilon(r))\rangle\Bigr] dr\\
&=\int_{s}^{t}\E\bigl[-\Ly\Phi\bigl(\Xb(r),Y^\epsilon(r),\theta_{s,t}(r)\bigr)\bigr]dr\\
&=\mathcal{I}_1^\epsilon(s,t)+\mathcal{I}_2^\epsilon(s,t)+\mathcal{I}_3^\epsilon(s,t),
\end{align*}
where
\begin{align}
\mathcal{I}_1^\epsilon(s,t)&=\epsilon \E\Phi(\Xb(s),Y^\epsilon(s),\theta_{s,t}(s))-\epsilon \E\Phi(\Xb(t),Y^\epsilon(t),\theta_{s,t}(t))\\
\mathcal{I}_2^\epsilon(s,t)&=-\epsilon\int_{s}^{t}\E\Bigl[\Phi\bigl(\Xb(r),Y^\epsilon(r),A\theta_{s,t}(r)\bigr)\Bigr]dr,\\\mathcal{I}_3^\epsilon(s,t)&=\epsilon \int_{s}^{t}\E\Bigl[\overline{\mathcal{L}}\Phi\bigl(\Xb(r),Y^\epsilon(r),\theta_{s,t}(r)\bigr)\Bigr]dr,
\end{align}
where $\overline{\mathcal{L}}$ is the infinitesimal generator associated with the averaged equation~\eqref{eq:av}, see~\eqref{eq:gen_av}.

For future use, a more detailed decomposition of the third term is introduced: $\mathcal{I}_3^\epsilon(s,t)=\mathcal{I}_{3,1}^\epsilon(s,t)+\mathcal{I}_{3,2}^\epsilon(s,t)+\mathcal{I}_{3,3}^\epsilon(s,t)$, with
\begin{align}
\mathcal{I}_{3,1}^\epsilon(s,t)&=\epsilon\int_{s}^{t}\E\langle \Fb(\Xb(r)),D_x\Phi\bigl(\Xb(r),Y^\epsilon(r),\theta_{s,t}(r)\bigr)\rangle dr\\
\mathcal{I}_{3,2}^\epsilon(s,t)&=\epsilon\int_{s}^{t}\E\langle A\Xb(r),D_x\Phi\bigl(\Xb(r),Y^\epsilon(r),\theta_{s,t}(r)\bigr)\rangle dr\\
\mathcal{I}_{3,3}^\epsilon(s,t)&=\frac{\epsilon}{2}\int_{s}^{t}\E{\rm Tr}\Bigl(QD_x^2\Phi\bigl(\Xb(r),Y^\epsilon(r),\theta_{s,t}(r)\bigr)\Bigr) dr
\end{align}

Lemmas~\ref{lem:I1},~\ref{lem:I2} and~\ref{lem:I3} below state the necessary estimates in order to conclude the analysis of the strong error. Observe that Assumption~\ref{ass:regular} is only used effectively in Lemma~\ref{lem:I3}.

\begin{lemma}\label{lem:I1}
There exists $C(T)\in(0,\infty)$, such that, for all $\epsilon\in(0,1)$,
\[
\underset{0\le s\le t\le T}\sup~|\mathcal{I}_1^\epsilon(s,t)|\le C(T)\epsilon(1+|x_0|_{L^2}^2+|y_0|_{L^2}^2).
\]
\end{lemma}

\begin{lemma}\label{lem:I2}
There exists $C(T)\in(0,\infty)$, such that, for all $\epsilon\in(0,1)$,
\[
\underset{0\le s<t\le T}\sup~(t-s)^{\frac12}|\mathcal{I}_2^\epsilon(s,t)|\le C(T)\epsilon(1+|x_0|_{L^2}^2+|y_0|_{L^2}^2).
\]
\end{lemma}

\begin{lemma}\label{lem:I3}
Let Assumption~\ref{ass:regular} be satisfied, and let $\gamma\in(1-\aly,\gamy)$ and $\kappa\in(0,\gamma)$. There exists $C_{\gamma,\kappa}(T)\in(0,\infty)$ such that, for all  $\epsilon\in(0,1)$,
\[
|\mathcal{I}_{3}^\epsilon(s,t)|\le C_{\gamma,\kappa}(T)(t-s)^{-\gamma-\frac{\kappa}{2}}\epsilon \bigl(1+|(-A)^{\gamma+\kappa}x_0|_{L^8}^3+|(-A)^{\gamma+\kappa}y_0|_{L^8}^3\bigr)\bigl(1+|(-A)^{1-\gamma}x_0|_{L^8}+M_{1-\gamma,8}(Q^\frac12)+{\rm Tr}(Q)\bigr).
\]
\end{lemma}

The proofs of the three auxiliary lemmas above is provided below, then the proof of Theorem~\ref{th:strong_regular} is concluded.

\begin{proof}[Proof of Lemma~\ref{lem:I1}]
For $r\in\left\{s,t\right\}$, note that $\E|\theta_{s,t}(r)|_{L^2}^2\le C(1+|x_0|_{L^2}^2+|y_0|_{L^2}^2)$, since $F$ has at most linear growth, and thanks to moment estimates, in Proposition~\ref{propo:well_av} and in Assumption~\ref{ass:Y}. Thus, thanks to Lemma~\ref{lem:Phi_0}.
\begin{align*}
\E|\Phi(\Xb(r),Y^\epsilon(r),\theta_{s,t}(r))|&\le C(1+(\E|Y^\epsilon(r)|_{L^2}^2)^{\frac12})(\E|\theta_{s,t}(r)|_{L^2}^2)^{\frac12}\\
&\le C(1+|y_0|_{L^2})(1+|x_0|_{L^2}+|y_0|_{L^2}),
\end{align*}
which concludes the proof.
\end{proof}

\begin{proof}[Proof of Lemma~\ref{lem:I2}]
For $r\in(s,t)$, using Lemma~\ref{lem:Phi_0}, and Assumption~\ref{ass:Y},
\begin{align*}
\E|\Phi\bigl(\Xb(r),Y^\epsilon(r),A\theta_{s,t}(r)\bigr)|&\le C(1+|y_0|_{L^2})(\E|A\theta_{s,t}(r)|_{L^2}^2)^{\frac12}\\
&\le C\|Ae^{(2t-s-r)A}\|_{\mathcal{L}(L^2,L^2)}(1+|x_0|_{L^2}^2+|y_0|_{L^2}^2).
\end{align*}
It is straightforward to check that $\int_{s}^{t}\|Ae^{(2t-s-r)A}\|_{\mathcal{L}(H)}dr\le C(T)\|(-A)^{\frac12}e^{(t-s)A}\|_{\mathcal{L}(L^2,L^2)}\le C(T)(t-s)^{-\frac12}$, thus one obtains
\[
(t-s)^{\frac12}|\mathcal{I}_2^\epsilon(s,t)|\le C(T)\epsilon(1+|x_0|_{L^2}^2+|y_0|_{L^2}^2),
\]
which concludes the proof.
\end{proof}

\begin{proof}[Proof of Lemma~\ref{lem:I3}]
First, proceeding as in the proof of Lemma~\ref{lem:I1}, and using the global Lipschitz continuity of $\Fb$, it is straightforward to check that
\[
|\mathcal{I}_{3,1}^\epsilon(s,t)|\le C_T\epsilon(1+|x_0|_{L^2}^3+|y_0|_{L^2}^3).
\]
Second, let $\gamma\in[0,\gamy)$, and $\kappa\in(\gamy-\gamma)$. Thanks to Lemma~\ref{lem:Phi_1bis}, and H\"older inequality, one obtains
\begin{align*}
|\mathcal{I}_{3,2}^\epsilon(s,t)|&\le C\epsilon\int_{s}^{t}(\E|(-A)^{1-\gamma}\Xb(r)|_{L^2}^4)^{\frac14}(\E|(-A)^{\gamma+\frac{\kappa}{2}}\theta_{s,t}(r)|_{L^4}^2)^{\frac12}\\
&\hspace{2cm}\bigl(1+(\E|(-A)^{\gamma+\kappa}\Xb(r)|_{L^8}^8)^{\frac14}+(\E|(-A)^{\gamma+\kappa}Y^\epsilon(r)|_{L^8}^8)^{\frac14}\bigr)dr\\
&\le C\epsilon (t-s)^{-\gamma-\frac{\kappa}{2}}(1+|x_0|_{L^4}+|y_0|_{L^4})\underset{r\in[0,T]}\sup~(\E|(-A)^{1-\gamma}\Xb(r)|_{L^2}^4)^{\frac14}\\
&\hspace{2cm}\bigl(1+|(-A)^{\gamma+\kappa}y_0|_{L^8}^2+\underset{r\in[0,T]}\sup~(\E|(-A)^{\gamma+\kappa}\Xb(r)|_{L^8}^8)^{\frac14}\bigr).
\end{align*}
Using the conditions on $\gamma$ and $\kappa$ above, and the moment estimates, one obtains
\[
|\mathcal{I}_{3,2}^\epsilon(s,t)|\le C_{\gamma,\kappa,T}\epsilon (t-s)^{-\gamma-\frac{\kappa}{2}}(1+|(-A)^{1-\gamma}x_0|_{L^8}^3+|(-A)^{\gamma+\kappa}y_0|_{L^8}^3)\bigl(1+|(-A)^{1-\gamma}x_0|_{L^8}+M_{1-\gamma,8}(Q^\frac12)\bigr).
\]
It remains to deal with the trace term, $\mathcal{I}_{3,3}^\epsilon$. Using Lemma~\ref{lem:Phi_2} and Assumption~\ref{ass:Q},
\begin{align*}
|\mathcal{I}_{3,3}^\epsilon(s,t)|&\le C\epsilon \int_{s}^{t}\sum_{n\in\N}q_n|D_x^2\Phi(\Xb(r),Y^\epsilon(r),\theta_{s,t}(r)).(f_n,f_n)|dr\\
&\le C\epsilon\sum_{n\in\N}q_n|f_n|_{L^4}^2 \int_{s}^{t}(\E|\theta_{s,t}(r)|_{L^4}^2)^{\frac12}(1+(\E|Y^\epsilon(r)|_{L^4}^2)^{\frac12})dr\\
&\le C\epsilon {\rm Tr}(Q)(1+|x_0|_{L^4}^2+|y_0|_{L^4}^2).
\end{align*}
Gathering the estimates on $|\mathcal{I}_{3,1}^\epsilon(s,t)|$, $|\mathcal{I}_{3,2}^\epsilon(s,t)|$ and $|\mathcal{I}_{3,3}^\epsilon(s,t)|$ then concludes the proof of Lemma~\ref{lem:I3}.
\end{proof}

Note that the assumption that ${\rm Tr}(Q)=\sum_{n\in\N}q_n$ is finite may be removed, using further regularity properties of the second order derivative $D_x^2\Phi$. However, this does not seem to improve the result.

\begin{proof}[Proof of Theorem~\ref{th:strong_regular}]
Gathering estimates from Lemmas~\ref{lem:I1}, \ref{lem:I2} and~\ref{lem:I3}, gives
\begin{align*}
\E\big|X^\epsilon(t)-\Xb(t)|_ {L^2}^2&\le CT\int_{0}^{t}\E\big|X^\epsilon(s)-\Xb(s)\big|_{L^2}^2ds+\int_{0}^{t}\bigl(|\mathcal{I}_1^\epsilon(s,t)|+|\mathcal{I}_2^\epsilon(s,t)|+|\mathcal{I}_3^\epsilon(s,t)|\bigr)ds\\
&\le CT\int_{0}^{t}\E\big|X^\epsilon(s)-\Xb(s)\big|_{L^2}^2ds++C(T,x_0,y_0)M_{1-\gamma,8}(Q^\frac12,T)\epsilon.
\end{align*}
It remains to apply Gronwall Lemma to conclude the proof.
\end{proof}

\section{Proof of Theorem~\ref{th:weak_regular}}\label{sec:proof_weak_regular}

The goal of this section is to provide a proof of Theorem~\ref{th:weak_regular}, {\it i.e.} that under Assumption~\ref{ass:regular}, the weak order of convergence in the averaging principle is equal to $1$.

Let $T\in(0,\infty)$. Thanks to Assumption~\ref{ass:regular}, let also $\gamma\in(1-\aly,\gamy)$, $\kappa\in(0,\gamy-\gamma)$, and let the initial conditions $x_0$ and $y_0$ satisfy $|(-A)^{1-\gamma}x_0|_{L^8}+|(-A)^{\gamma+\kappa}y_0|_{L^8}<\infty$.

A key tool in the analysis is the function $\ub$ defined below:
\begin{equation}\label{eq:ub}
\ub(t,x)=\E[\varphi(\Xb^x(t))].
\end{equation}
Note that $\ub$ is the solution of the Kolmogorov equation
\[
\partial_t\ub=\Lb \ub,
\]
with initial condition $\ub(0,\cdot)=\varphi$, where $\Lb$ is the infinitesimal generator associated with the averaged equation~\eqref{eq:av}, see~\eqref{eq:gen_av}.

To deal with this infinite dimensional PDE, usually an auxiliary approximation procedure is employed, see for instance~\cite{BD}, in order to justify the computations. To simplify notation, this is omitted in this manuscript.

The regularity properties stated in Proposition~\ref{propo:ub} play a fundamental role in the analysis of the weak error below.
\begin{propo}[Regularity properties of the derivatives of $\ub$]\label{propo:ub}
Let $\varphi$ be an admissible test function.

For all $\beta<1$, there exists $C_\beta(T)\in(0,\infty)$, such that for all $t\in(0,T]$,
\begin{equation}\label{eq:ub_1}
|D_x\ub(t,x).(-A)^\beta h|\le C_\beta(T)t^{-\beta}|h|_{L^2},
\end{equation}
and for all $\beta_1,\beta_2\in(0,1)$ such that $\beta_1+\beta_2<1$, there exists $C_{\beta_1,\beta_2}(T)\in(0,\infty)$, such that for all $t\in(0,T]$,
\begin{equation}\label{eq:ub_2}
\big|D_x^2\ub(t,x).\bigl((-A)^{\beta_1}h_1,(-A)^{\beta_2}h_2\bigr)\big|\le C_{\beta_1,\beta_2}(T)t^{-\beta_1-\beta_2}|h_1|_{L^{2}}|h_2|_{L^{2}}.
\end{equation}
In addition, for $p_1,p_2,p_3\in[2,\infty)$ such that $1=\frac{1}{p_1}+\frac{1}{p_2}+\frac{1}{p_3}$, there exists $C_{p_1,p_2,p_3}\in(0,\infty)$, such that for all $t\in[0,T]$,
\begin{equation}\label{eq:ub_3}
|D_x^3\ub(t,x).(h_1,h_2,h_3)|\le C_{p_1,p_2,p_3}(T)|h_1|_{L^{p_1}}|h_2|_{L^{p_2}}|h_3|_{L^{p_3}}.
\end{equation}
\end{propo}

Regularity properties for infinite dimensional Kolmogorov equations, as stated in Proposition~\ref{propo:ub}, are now a classical tool in the analysis of parabolic SPDEs. We refer to~\cite{BD} for a recent overview of this topic and for further results. A sketch of proof is provided at the end of this section.

For the analysis of the averaging error, in the weak sense, the fundamental object is the auxiliary function $v$ defined by
\begin{equation}\label{eq:v}
v(t,x,y)=\Phi\bigl(x,y,D_x\ub(t,x)\bigr),
\end{equation}
where the first order derivative $D_x\ub(t,x)$ is interpreted as an element of $L^2$.

By construction, $v(t,x,\cdot)$ is the solution of the Poisson equation~\eqref{eq:Poisson} with $\theta=D_x\ub(t,x)$, {\it i.e.} one has the fundamental identity
\begin{equation}\label{eq:v_identity}
-\Ly v(t,x,y)=\langle F(x,y)-\Fb(x),D_x\ub(t,x)\rangle.
\end{equation}

For all $y\in L^2$, denote by $\mathcal{L}_y$ is the infinitesimal generator given by
\[
\mathcal{L}_y\varphi(x)=\langle D_x\varphi(x),Ax+F(x,y)\rangle+\frac12 \sum_{n\in\N}q_nD_x^2\varphi(x,y).\bigl(f_n,f_n\bigr),
\]
for functions $\varphi:x\in L^2\mapsto \varphi(x)\in\R$, depending only on the slow variable $x$.

Applying It\^o formula, the weak error is written as
\begin{align*}
\E[\varphi(X^\epsilon(T))]-\E[\varphi(\Xb(T))]&=\E[\ub(0,X^\epsilon(T))]-\E[\ub(T,X^\epsilon(0))]\\
&=\int_{0}^{T}\E\bigl[-\partial_t\ub(T-t,X^\epsilon(t))+\mathcal{L}_{Y^\epsilon(t)}\ub(T-t,X^\epsilon(t))\bigr]dt\\
&=\int_{0}^{T}\E\bigl[(\mathcal{L}_{Y^\epsilon(t)}-\Lb)\ub(T-t,X^\epsilon(t))\bigr]dt\\
&=\int_{0}^{T}\E\bigl[\langle F(X^\epsilon(t),Y^\epsilon(t))-\overline{F}(X^\epsilon(t)),D_x\overline{u}(T-t,X^\epsilon(t))\rangle\bigr]dt\\
&=\int_{0}^{T}\E\bigl[-\Ly v(T-t,X^\epsilon(t),Y^\epsilon(t))\bigr]dt,
\end{align*}
thanks to the identity~\eqref{eq:v_identity}. To exploit this formula for the weak error, note that the It\^o formula applied with the function $v$ yields the identity
\begin{align*}
\E[v(0,X^\epsilon(T),Y^\epsilon(T))]&=\E[v(T,X^\epsilon(0),Y^\epsilon(0)]\\
&+\int_{0}^{T}\E\bigl[\bigl(\mathcal{L}_{Y^\epsilon(t)}+\frac{1}{\epsilon}\Ly-\partial_t\bigr)v(T-t,X^\epsilon(t),Y^\epsilon(t))\bigr]dt.
\end{align*}
As a consequence, the weak error has the following decomposition
\begin{equation}\label{eq:decomp_weak}
\E[\varphi(X^\epsilon(T))]-\E[\varphi(\Xb(T))]=\mathcal{J}_1^\epsilon+\mathcal{J}_2^\epsilon+\mathcal{J}_3^\epsilon,
\end{equation}
where
\begin{align*}
\mathcal{J}_1^\epsilon&=\epsilon\bigl(\E[v(T,X^\epsilon(0),Y^\epsilon(0)]-\E[v(0,X^\epsilon(T),Y^\epsilon(T))]\bigr)\\
\mathcal{J}_2^\epsilon&=-\epsilon\int_{0}^{T}\E\bigl[\partial_t v(T-t,X^\epsilon(t),Y^\epsilon(t))\bigr]dt\\
\mathcal{J}_3^\epsilon&=\epsilon\int_{0}^{T}\E\bigl[\mathcal{L}_{Y^\epsilon(t)}v(T-t,X^\epsilon(t),Y^\epsilon(t))\bigr]dt,
\end{align*}
and the third expresion is decomposed as $\mathcal{J}_3^\epsilon=\mathcal{J}_{3,1}^\epsilon+\mathcal{J}_{3,2}^\epsilon+\mathcal{J}_{3,3}^\epsilon$, where
\begin{align*}
\mathcal{J}_{3,1}^\epsilon&=\epsilon\int_{0}^{T}\E\bigl[\langle F(X^\epsilon(t),Y^\epsilon(t)),D_xv(T-t,X^\epsilon(t),Y^\epsilon(t))\rangle\bigr]dt\\
\mathcal{J}_{3,2}^\epsilon&=\epsilon\int_{0}^{T}\E\bigl[\langle AX^\epsilon(t),D_xv(T-t,X^\epsilon(t),Y^\epsilon(t))\rangle\bigr]dt\\
\mathcal{J}_{3,3}^\epsilon&=\frac{\epsilon}{2}\int_{0}^{T}\E\bigl[\sum_{n\in\N}q_nD_x^2 v(T-t,X^\epsilon(t),Y^\epsilon(t)).\bigl(f_n,f_n\bigr)\bigr]dt.
\end{align*}

Theorem~\ref{th:weak_regular} is then a straightforward consequence of the three auxiliary results stated below.

\begin{lemma}\label{lem:J1}
There exists $C(T)\in(0,\infty)$, such that, for all $\epsilon\in(0,1)$, and all $x_0,y_0\in H$,
\[
|\mathcal{J}_1^\epsilon|\le C(T)\epsilon(1+|y_0|_{L^2}).
\]
\end{lemma}

\begin{lemma}\label{lem:J2}
Let $\kappa\in(0,\gamy)$. There exists $C_\kappa(T)\in(0,\infty)$, such that, for all $\epsilon\in(0,1)$, and all $x_0,y_0\in H$,
\[
|\mathcal{J}_2^\epsilon|\le C_\kappa(T)\epsilon\bigl(1+|(-A)^{2\kappa}x_0|_{L^4}^2+|(-A)^{2\kappa}y_0|_{L^4}^2\bigr).
\]
\end{lemma}

\begin{lemma}\label{lem:J3}
Let $\gamma\in(1-\aly,\gamy)$ and $\kappa\in(0,\gamy-\gamma)$. There exists $C_{\gamma,\kappa}(T)\in(0,\infty)$ such that, for all $\epsilon\in(0,1)$, and all $x_0,y_0\in L^8$,
\[
|\mathcal{J}_{3}^\epsilon|\le C_{\gamma,\kappa}(T)\epsilon\bigl(1+|(-A)^{\gamma+\kappa}x_0|_{L^8}^2+|(-A)^{\gamma+\kappa}y_0|_{L^8}^2\bigr)\bigl(1+|(-A)^{1-\gamma}x_0|_{L^4}+{\rm Tr}(Q)+M_{\alpha,4}(Q^\frac12,T)\bigr).
\]
\end{lemma}
Note that Assumption~\ref{ass:regular} is only required in Lemma~\ref{lem:J3}.

\begin{proof}[Proof of Lemma~\ref{lem:J1}]
Thanks to Lemma~\ref{lem:Phi_0} and Proposition~\ref{propo:ub}, for all $t\in[0,T]$, $x,y\in L^2$,
\[
|v(t,x,y)|=|\Phi(x,y,D_x\ub(t,x)|\le C(1+|y|)|D_x\ub(t,x)|_{L^2}\le C(T,\varphi)(1+|y|_{L^2}).
\]
Combined with Assumption~\ref{ass:Y}, this estimate concludes the proof of Lemma~\ref{lem:J1}.
\end{proof}

\begin{proof}[Proof of Lemma~\ref{lem:J2}]
Since the mapping $\theta\in H\mapsto \Phi(x,y,\theta)$ is a continuous linear mapping (thanks to Lemma~\ref{lem:Phi_0}), one has the following expression,
\begin{align*}
\partial_tv(t,x,y)&=\Phi\bigl(x,y,\partial_tD_x\overline{u}(t,x)\bigr)\\
&=\Phi\bigl(x,y,D_x\partial_t\overline{u}(t,x)\bigr)\\
&=\Phi\bigl(x,y,D_x\bigl(\Lb\overline{u}(t,x)\bigr)\bigr)\\
&=\Phi\bigl(x,y,\Theta_1(t,x)\bigr)+\Phi\bigl(x,y,\Theta_2(t,x)\bigr)+\Phi\bigl(x,y,\Theta_3(t,x)\bigr),
\end{align*}
where
\begin{align}
\langle \Theta_1(t,x),h\rangle&=\langle Ah+D\Fb(x).h,D_x\ub(t,x)\rangle,\\
\langle \Theta_2(t,x),h\rangle&=D_x^2\ub(t,x).(h,Ax+\Fb(x)),\\
\langle \Theta_3(t,x),h\rangle&=\frac12\sum_{n\in\N}q_nD_x^3\ub(t,x).(f_n,f_n,h).
\end{align}

Let $\kappa\in(0,1)$. Thanks to~\eqref{eq:ub_1}, one has
\[
|\langle \Theta_1(t,x),h\rangle|\le C_\kappa(T)t^{-1+\kappa}|(-A)^{\kappa}h|_{L^2},
\]
which implies $|(-A)^{-\kappa}\Theta_1(t,x)|_{L^2}\le C_\kappa(T)t^{-1+\kappa}$. Thus, thanks to Lemma~\ref{lem:Phi_0ter},
\[
\big|\Phi\bigl(x,y,\Theta_1(t,x)\bigr)\big|\le C_\kappa(T)t^{-1+\kappa}\bigl(1+|(-A)^{2\kappa}x|_{L^4}^2+|(-A)^{2\kappa}y|_{L^4}^2\bigr).
\]
Thanks to~\eqref{eq:ub_2}, one has
\[
|\Theta_2(t,x)|_{L^{2}}=\underset{h\in L^2,|h|_{L^2}\le 1}\sup~|\langle \Theta_2(t,x),h\rangle|\le C_{\kappa}(T)t^{-1+\kappa}(1+|(-A)^{\kappa}x|_{L^2}).
\]
Thus, thanks to Lemma~\ref{lem:Phi_0},
\[
\big|\Phi\bigl(x,y,\Theta_2(t,x)\bigr)|\le C_\kappa(T)t^{-1+\kappa}(1+|(-A)^\kappa x|_{L^2}^2+|y|_{L^2}^2).
\]
Finally, thanks to~\eqref{eq:ub_3} and Assumption~\ref{ass:Q}, one has, for all $h\in L^4$,
\[
|\langle\Theta_3(t,x),h\rangle|\le C(T)\sum_{n\in\N}q_n|f_n|_{L^8}^2|h|_{L^4}\le C(T)|h|_{L^4},
\]
{\it i.e.} $|\Theta_3(t,x)|_{L^{\frac43}}=\underset{h\in L^4,|h|_{L^4}\le 1}\sup~|\langle\Theta_3(t,x),h\rangle|\le C(T)$. Thus Lemma~\ref{lem:Phi_0} yields
\[
\big|\Phi\bigl(x,y,\Theta_3(t,x)\bigr)|\le C(T)(1+|y|_{L^4}).
\]
Gathering the above estimates then yields, if $2\kappa<\gamy$,
\begin{align*}
|\mathcal{J}_2^\epsilon|&\le C_{\kappa}(T)\epsilon\int_{0}^{T}(1+t^{-1+\kappa})\E\bigl(1+|(-A)^{2\kappa}X^{\epsilon}(t)|_{L^4}^2+|(-A)^{2\kappa}Y^\epsilon(t)|_{L^4}^2\bigr)dt\\
&\le C_\kappa(T)T^\kappa \bigl(1+|(-A)^{2\kappa}x_0|_{L^4}^2+|(-A)^{2\kappa}y_0|_{L^4}^2\bigr).
\end{align*}
This concludes the proof of Lemma~\ref{lem:J2}.
\end{proof}

\begin{proof}[Proof of Lemma~\ref{lem:J3}]

Note that the first-order derivative of $v$ with respect to $x$ satisfies the following identity:
\[
\langle D_xv(t,x,y),h\rangle=\Phi\bigl(x,y,D_x^2\ub(t,x).(h,\cdot)\bigr)+\langle D_x\Phi\bigl(x,y,D_x\ub(t,x)\bigr),h\rangle.
\]
Observe that $|D_x^2\ub(t,x).(h,\cdot)|_{L^{2}}\le C|h|_{L^2}$, thanks to~\eqref{eq:ub_2}, with $\beta_1=\beta_2=0$. Then, thanks to Lemma~\ref{lem:Phi_0} and Lemma~\ref{lem:Phi_1}, one  obtains
\begin{align*}
|\langle D_xv(t,x,y),h\rangle|&\le C(1+|y|_{L^2})|D_x^2\ub(t,x).(h,\cdot)|_{L^{2}}+C(1+|y|_{L^4})|D_x\ub(t,x)|_{L^2}|h|_{L^4}\\
&\le C(1+|y|_{L^4})|h|_{L^4}.
\end{align*}

Since $F$ has at most linear growth, using moment estimates then yields
\[
\E|\mathcal{J}_{3,1}^\epsilon|\le C(T)\epsilon (1+|x_0|_{L^4}^2+|y_0|_{L^4}^2). 
\]

To treat the second term, $\mathcal{J}_{3,2}^\epsilon$, observe that $|D_x^2\ub(t,x).(h,\cdot)|_{L^{2}}\le C_{\kappa}t^{-1+\kappa}|(-A)^{-1+\kappa}h|_{L^2}$, for all $\kappa\in(0,1]$, thanks to~\eqref{eq:ub_2}. In addition, $|(-A)^{1-\kappa}D_x\ub(t,x)|_{L^2}\le C_\kappa t^{-1+\kappa}$, thanks to~\eqref{eq:ub_1}. Then, thanks to Lemma~\ref{lem:Phi_0} and Lemma~\ref{lem:Phi_1bis},
\begin{align*}
|\langle D_xv(t,x,y),h\rangle|&\le C_\kappa(1+|y|_{L^2})t^{-1+\kappa}|(-A)^{-1+\kappa}h|_{L^2}\\
&+C_{\gamma,\kappa}\bigl(1+|(-A)^{\gamma+\kappa}x|_{L^8}^2+|(-A)^{\gamma+\kappa}y|_{L^8}^2\bigr)t^{-\gamma-\frac{\kappa}{2}}|(-A)^{-\gamma}h|_{L^4},
\end{align*}
where $\gamma<\gamy$ and $\kappa\in(0,\gamy-\gamma)$.

As a consequence
\begin{align*}
|\mathcal{J}_{3,2}^\epsilon|&\le C_\kappa\epsilon\int_{0}^{T}\bigl(1+\E|Y^\epsilon(t)|_{L^2}^2\bigr)^{\frac12}\bigl(\E|(-A)^{\kappa}X^\epsilon(t)|_{L^2}^2\bigr)^{\frac12}(T-t)^{-1+\kappa}dt\\
&+C_{\gamma,\kappa}\epsilon\int_{0}^{T}\bigl(1+\E|(-A)^{\gamma+\kappa}X^\epsilon(t)|_{L^8}^4+|(-A)^{\gamma+\kappa}Y^\epsilon(t)|_{L^8}^4\bigr)^{\frac12}\bigl(\E|(-A)^{1-\gamma}X^\epsilon(t)|_{L^4}^2\bigr)^{\frac12}(T-t)^{-\gamma-\frac{\kappa}{2}}dt.
\end{align*}
It remains to use the condition that $\gamma\in(1-\aly,\gamy)$, thanks to Assumption~\ref{ass:regular}. Note that $\gamma+\kappa\le \gamy \le \frac12\le 1-\gamma$. Finally, thanks to moment estimates,
\[
|\mathcal{J}_{3,2}^\epsilon|\le C_{\gamma,\kappa}\epsilon\bigl(1+|(-A)^{\gamma+\kappa}x_0|_{L^8}^2+|(-A)^{\gamma+\kappa}y_0|_{L^8}^2\bigr)\bigl(1+|(-A)^{1-\gamma}x_0|_{L^4}+M_{1-\gamma,4}(Q^\frac12,T)\bigr).
\]

It remains to deal with the third term, $\mathcal{J}_{3,3}^{\epsilon}$. Note that the second-order derivative of $v$ with respect to $x$ satisfies the identity
\begin{align*}
D_x^2v(t,x,y).(h,h)&=\Phi\bigl(x,y,D_x^3\ub(t,x).(h,h,\cdot)\bigr)+2\langle D_x\Phi\bigl(x,y,D_x^2\ub(t,x).(h,\cdot)\bigr),h\rangle\\
&~+D_x^2\Phi\bigl(x,y,D_x\ub(t,x)\bigr).(h,h).
\end{align*}

First, observe that $|D_x^3\ub(t,x).(h,h,k)|\le C|h|_{L^4}^2|k|_{L^2}$, thanks to~\eqref{eq:ub_3}. Equivalently, this means that $|D_x^3\ub(t,x).(h,h,\cdot)|_{L^{2}}\le C|h|_{L^4}^2$, then Lemma~\ref{lem:Phi_0} yields
\[
\big|\Phi\bigl(x,y,D_x^3\ub(t,x).(h,h,\cdot)\bigr)\big|\le C(1+|y|_{L^2})|h|_{L^4}^2.
\]
Second, thanks to Lemma~\ref{lem:Phi_1},
\begin{align*}
\big|\langle D_x\Phi\bigl(x,y,D_x^2\ub(t,x).(h,\cdot)\bigr),h\rangle\big|&\le C(1+|y|_{L^4})|h|_{L^4}|D_x^2\ub(t,x).(h,\cdot)|_{L^{2}}\\
&\le C(1+|y|_{L^4})|h|_{L^4}^2.
\end{align*}
Finally, Lemma~\ref{lem:Phi_2} and~\eqref{eq:ub_1} yield
\begin{align*}
\big|D_x^2\Phi\bigl(x,y,D_x\ub(t,x)\bigr).(h,h)\big|&\le C(1+|y|_{L^4})|D_x\ub(t,x)|_{L^2}|h|_{L^8}^2\le C(1+|y|_{L^4})|h|_{L^4}^2.
\end{align*}

As a consequence, one obtains
\begin{align*}
|\mathcal{J}_{3,3}^\epsilon|&=\Big|\frac{\epsilon}{2} \int_{0}^{T}\sum_{n\in\N}q_n \E\bigl[D_x^2v(T-t,X^\epsilon(t),Y^\epsilon(t)).(f_n,f_n)\bigr]dt\Big|\\
&\le C\epsilon \sum_{n\in \N}q_n |f_n|_{L^4}^2 \int_{0}^{T}\bigl(1+\E|Y^\epsilon(t)|_{L^4}\bigr)dt\\
&\le C(T){\rm Tr}(Q)\epsilon(1+|y_0|_{L^4}),
\end{align*}
thanks to Assumption~\ref{ass:Q}, and a moment estimate, see Assumption~\ref{ass:Y}.

Gathering the estimates for $\mathcal{J}_{3,1}^\epsilon$, $\mathcal{J}_{3,2}^\epsilon$ and $\mathcal{J}_{3,3}^\epsilon$, one obtains
\[
|\mathcal{J}_{3}^\epsilon|\le C_{\gamma,\kappa}\epsilon\bigl(1+|(-A)^{1-\gamma}x_0|_{L^8}^4+|(-A)^{\gamma+\kappa}y_0|_{L^8}^4\bigr).
\]
This concludes the proof of Lemma~\ref{lem:J3}.
\end{proof}

We are now in position to conclude the proof of Theorem~\ref{th:weak_regular}.
\begin{proof}
Thanks to the decomposition~\eqref{eq:decomp_weak} of the weak error, it suffices to gather the estimates of Lemmas~\ref{lem:J1},~\ref{lem:J2} and~\ref{lem:J3} to conclude.
\end{proof}

To conclude this section, we provide a sketch of proof of Proposition~\ref{propo:ub}, which states the regularity properties for the spatial derivatives of $\ub(t,x)$ used above.

\begin{proof}[Sketch of proof of Proposition~\ref{propo:ub}]
The proof is based on computing the derivatives of $\ub$ in terms of tangent processes, which are solutions of PDEs with random coefficients (noise is additive in~\eqref{eq:av}). A mild formulation and regularity properties of the semigroup $(e^{tA})_{t\ge 0}$ yield the required estimates.
\begin{itemize}
\item First-order derivative.

Note that
\[
D_x\ub(t,x).h=\E\bigl[D\varphi(\Xb^x(t)).\eta^h(t)\bigr],
\]
where
\[
\eta^h(t)=e^{tA}h+\int_{0}^{t}e^{(t-s)A}D\Fb(\Xb^x(s)).\eta^h(s)ds.
\]
Thanks to the global Lipschitz continuity of the averaged coefficient $\Fb$, one obtains, for all $t\in(0,T]$,
\[
|\eta^h(t)|_{L^2}\le C_{p,\beta} t^{-\beta}|(-A)^{-\beta}h|_{L^2}+\int_{0}^{t}|\eta^h(s)|_{L^2}ds.
\]
An application of Gronwall Lemma yields
\[
|\eta^h(t)|_{L^2}\le C_{\beta}(T)t^{-\beta}|(-A)^{-\beta}h|_{L^2},
\]
hence~\eqref{eq:ub_1}


\item Second-order derivative.

Note that
\begin{align*}
D_x^2\ub(t,x).(h_1,h_2)&=\E\bigl[\langle D\varphi(\Xb^x(t)),\zeta^{h_1,h_2}(t)\rangle\bigr]\\
&~+\E\bigl[D^2\varphi(\Xb^x(s)).(\eta^{h_1}(t),\eta^{h_2}(t))\bigr],
\end{align*}
where
\[
\zeta^{h_1,h_2}(t)=\int_{0}^{t}e^{(t-s)A}D\Fb(\Xb^x(s)).\zeta^{h_1,h_2}(s)ds+\int_{0}^{t}e^{(t-s)A}D^2\Fb(\Xb^x(s)).(\eta^{h_1}(s),\eta^{h_2}(s))ds.
\]
First, $\varphi$ is an admissible test function, thus one obtains
\[
\E\bigl[D^2\varphi(\Xb^x(s)).(\eta^{h_1}(t),\eta^{h_2}(t))\bigr]\le C\E\bigl[|\eta^{h_1}(t)|_{L^2}|\eta^{h_2}(t)|_{L^2}\bigr],
\]
which is treated using the estimate proved above.

To treat the second term, the inequality $\|e^{tA}\|_{\mathcal{L}(L^1,L^2)}\le C_d t^{-\frac{d}{4}}$ is used. This may be proved as follows. First, by a duality argument, $\|e^{tA}\|_{\mathcal{L}(L^1,L^2)}=\|e^{tA}\|_{\mathcal{L}(L^2,L^\infty)}$. In addition, by Jensen inequality, one has
\[
e^{tA}x(\xi)^2=\bigl(\int K(t,\xi,\eta)x(\eta)d\eta\bigr)^2\le \int K(t,\xi,\eta)x(\eta)^2d\eta \le C_dt^{-\frac{d}{2}}|x|_{L^2}^2,
\]
where $K$ is the kernel associated with the semigroup. As a consequence, one has $\|e^{tA}\|_{\mathcal{L}(L^2,L^\infty)}\le C_dt^{-\frac{d}{4}}$. Note that $\frac{d}{4}<1$ if $d\in\left\{1,2,3\right\}$.

Thanks to~\eqref{eq:propoFb1} and~\eqref{eq:propoFb2}, and the estimate above, one obtains, with the application of Gronwall Lemma,
\begin{align*}
|\zeta^{h_1,h_2}(t)|_{L^2}&\le C_p\int_{0}^{t}|\zeta^{h_1,h_2}(s)|_{L^2}ds+C_{\beta_1,\beta_2}\int_{0}^{t}(t-s)^{-\frac{d}{4}}s^{-\beta_1-\beta_2}ds|(-A)^{-\beta_1}h_1|_{L^2}|(-A)^{-\beta_2}h_2|_{L^2}\\
&\le C_{\beta_1,\beta_2}(T)|(-A)^{-\beta_1}h_1|_{L^{2}}|(-A)^{-\beta_2}h_2|_{L^{2}}.
\end{align*}
It is then straightforward to obtain~\eqref{eq:ub_2}.

\item Third-order derivative: the proof is omitted, since the computations are similar.
\end{itemize}
\end{proof}

\section{Proof of Theorem~\ref{th:general}}\label{sec:proof_general}

This section is devoted to the proof of Theorem~\ref{th:general}. Let Assumption~\ref{ass:general} be satisfied. First, let us justify the definition of
\[
\aly=\frac{1}{2}\bigl(1-\frac{d}{2}(1-\frac{2}{\varrho})\bigr).
\]

\begin{propo}\label{propo:suffQ}
Let Assumption~\ref{ass:general} be satisfied. Then~\eqref{eq:assQ} is satisfied: for all $\alpha\in[0,\aly)$ and all $p\ge 2$,
\[
M_{\alpha,p}(Q^{\frac12},T)<\infty.
\]
\end{propo}

\begin{proof}[Proof of Proposition~\ref{propo:suffQ}]
Let $\varsigma=\frac{\varrho}{\varrho-2}>1$, and note that $1=\frac{2}{\varrho}+\frac{1}{\varsigma}$.

Using the ideal property for $\gamma$-Radonifying operators,
\begin{align*}
\int_0^T\|e^{tA}(-A)^\alpha Q^{\frac12}\|_{\mathcal{R}(L^2,L^p)}^2dt&\le \int_0^T\|e^{\frac{t}{2}A}(-A)^{\alpha}\|_{\mathcal{L}(L^p,L^p)}^2\|e^{\frac{t}{2}A}Q^{\frac12}\|_{\mathcal{R}(L^2,L^p)}^2dt\\
&\le C_{\alpha,p}\int_0^T t^{-2\alpha} \big|\sum_{n\in\N}q_n (e^{tA}f_n)^2\big|_{L^{\frac{p}{2}}}dt.
\end{align*}

Using H\"older inequality, for all $\xi\in\mathcal{D}$, and all $t>0$,
\begin{align*}
\sum_{n\in\N}q_n\bigl(e^{tA}f_n\bigr)^2(\xi)&\le \bigl(\sum_{n\in\N}q_n^{\frac{\varrho}{2}}\bigr)^{\frac{2}{\varrho}} \Bigl(\sum_{n\in\N}\bigl(e^{tA}f_n\bigr)^{2\varsigma}(\xi)\Bigr)^{\frac{1}{\varsigma}}\\
&\le C(Q)\Bigl(\underset{k\in\N}\sup~\bigl(e^{tA}f_k\bigr)^{\frac{2(\varsigma-1)}{\varsigma}}(\xi)\Bigr)\Bigl(\sum_{n\in\N}\bigl(e^{tA}f_n\bigr)^{2}(\xi)\Bigr)^{\frac{1}{\varsigma}}.
\end{align*}

Recall that $K$ is the kernel associated with the semigroup $\bigl(e^{tA}\bigr)_{t\ge 0}$.

On the one hand, Assumption~\ref{ass:Q} implies that for all $z\in\mathcal{D}$,
\[
\underset{k\in\mathbb{N}}\sup~|e^{tA}f_k (\xi)|\le \int_{\mathcal{D}}K(t,z,\cdot) \underset{k\in\mathbb{N}}\sup~|f_k|_{L^\infty}\le C.
\]
On the other hand, $\bigl(f_n\bigr)_{n\in\mathbb{N}}$ is a complete orthonormal system of $L^2$, hence
\begin{align*}
\sum_{n\in\N}\bigl(e^{tA}f_n\bigr)^{2}(\xi)&=\sum_{n\in\N}\langle K(t,\xi,\cdot),f_n\rangle^2=|K(t,\xi,\cdot)|_{L^2}^2\\
&=\int_{\mathcal{D}}K(t,z,\eta)^2d\eta\\
&\le Ct^{-\frac{d}{2}}\int_{\mathcal{D}}K(t,\xi,\eta)d\eta=Ct^{-\frac{d}{2}},
\end{align*}
using the properties of the kernel $K$ stated above.

Finally, for all $t>0$ and all $z\in\mathcal{D}$, one obtains
\[
\big|\sum_{n\in\N}q_n (e^{tA}f_n)^2\big|_{L^{\frac{p}{2}}}\le Ct^{-\frac{d}{2\varsigma}},
\]
thus
\[
\int_0^T\|e^{tA}(-A)^\alpha Q^{\frac12}\|_{\mathcal{R}(L^2,L^p)}^2dt\le C\int_0^T t^{-2\alpha-\frac{d}{2\varsigma}}dt.
\]
It remains to check that $2\alpha-\frac{d}{2\varsigma}=2\alpha-\frac{d}{2}(1-\frac{2}{\varrho})<1$ for $\alpha<\aly=\frac{1}{2}\bigl(1-\frac{d}{2}(1-\frac{2}{\varrho})\bigr)$.

This concludes the proof of Proposition~\ref{propo:suffQ}. 
\end{proof}

The approximation argument is based on the following estimate.
\begin{lemma}\label{lem:regul}
Let Assumption~\ref{ass:general} be satisfied. For all $\alpha\in[0,\aly)$, $\gamma\in[0,\gamy)$, such that $\alpha\ge \gamma$ and $\alpha+\gamma\le 1$, all $T\in(0,\infty)$ and $p\ge 2$, there exists $C_{\alpha,\gamma,p}(Q,T)\in(0,\infty)$, such that for all $\delta>0$,
\begin{equation}
{\rm Tr}(e^{2\delta A}Q)+M_{1-\gamma,p}(e^{\delta A}Q^{\frac12},T)\le C_{\alpha,\gamma,p}(Q,T)\delta^{\alpha+\gamma-1}.
\end{equation}
\end{lemma}

\begin{proof}[Proof of Lemma~\ref{lem:regul}]
First, note that
\[
M_{1-\gamma,p}(e^{\delta A}Q^{\frac12},T)\le \|(-A)^{1-\gamma-\alpha}e^{\delta A}\|_{\mathcal{L}(L^p,L^p)}M_{\alpha,p}(Q^{\frac12},T),
\]
and that $\|(-A)^{1-\gamma+\alpha}e^{\delta A}\|_{\mathcal{L}(L^p,L^p)}\le C_{\alpha,\gamma}\delta t^{\gamma+\alpha-1}$, in the regime $\alpha+\gamma\le 1$.

To deal with the trace term, we use the H\"older type inequality for Schatten norms $\|\cdot\|_{\mathcal{L}_\varrho(L^2)}$, with parameter $\varrho\in[1,\infty]$, see for instance~\cite[Corollary~D.2.4, Appendix D]{HytonenVanNeervenVeraar1}. One obtains
\begin{align*}
{\rm Tr}(e^{2\delta A}Q)&=\|e^{2\delta A}Q\|_{\mathcal{L}_1(L^2)}\le \|e^{2\delta tA}\|_{\mathcal{L}_{\varsigma}(L^2)}\|Q\|_{\mathcal{L}_{\frac{\varrho}{2}}(L^2)},
\end{align*}
where $1=\frac{2}{\varrho}+\frac{1}{\varsigma}$. By assumption, $\|Q\|_{\mathcal{L}_{\frac{\varrho}{2}}(L^2)}<\infty$. In addition,
\[
\|e^{2\delta tA}\|_{\mathcal{L}_{\varsigma}(L^2)}^{\varsigma}\le \sum_{n\in\N}|e^{2\delta A}e_n|_{L^2}^{\varsigma}\le \sum_{n\in\N}e^{-2\varsigma \delta \lambda_n}\le C_\varsigma \delta^{-\frac{d}{2}},
\]
using $\lambda_n\underset{n\to\infty}\sim c_d n^{\frac{2}{d}}$. As a consequence,
\[
{\rm Tr}(e^{2\delta A}Q)\le C_\varsigma \delta^{-\frac{d}{2\varsigma}}=C_\varsigma \delta^{2\aly-1}\le C_\varsigma \delta^{2\alpha-1},
\]
using the definition of $\aly=\frac{1}{2}\bigl(1-\frac{d}{2}(1-\frac{2}{\varrho})\bigr)=\frac{1}{2}\bigl(1-\frac{d}{2\varsigma}\bigr)$.

Finally, one concludes using $2\alpha-1\le \alpha+\gamma-1\le 0$.
\end{proof}

The result of Lemma~\ref{lem:regul} motivates the introduction of the following auxiliary SPDE problems, where $Q^{\frac12}$ is replaced by $e^{\delta A}Q^\frac12$. For all $\delta>0$ (this parameter will be chosen below), $X_{\delta}^{\epsilon}$ and $\Xb_\delta$ are solutions of
\begin{equation}\label{eq:SPDE_delta}
\begin{aligned}
dX_{\delta}^{\epsilon}(t)&=AX_{\delta}^{\epsilon}(t)dt+F\bigl(X_{\delta}^{\epsilon}(t),Y^\epsilon(t)\bigr)dt+e^{\delta A}dW^Q(t),\\
d\Xb_\delta(t)&=A\Xb_\delta(t)dt+\Fb(\Xb(t))dt+e^{\delta A}dW^Q(t),
\end{aligned}
\end{equation}
with initial conditions $X_\delta^\epsilon(0)=\Xb_\delta=x_0$.

Then Theorem~\ref{th:general} follows from Lemmas~\ref{lem:general1} and~\ref{lem:general2} stated below.

First, thanks to Lemma~\ref{lem:regul}, the strong and weak convergence results, with orders $\frac12$ and $1$, from Theorems~\ref{th:strong_regular} and~\ref{th:weak_regular} may be applied when considering the auxiliary processes $X_\delta^\epsilon$ and $\Xb_\delta$ defined by~\eqref{eq:SPDE_delta}.
\begin{lemma}\label{lem:general1}
Let Assumption~\ref{ass:general} be satisfied. Let $T\in(0,\infty)$, and assume that the initial conditions $x_0$, $y_0$, satisfy
\[
|(-A)^{1-\gamma}x_0|_{L^8}+|(-A)^{\gamma+\kappa}y_0|_{L^8}<\infty,
\]
with $\gamma\in(0,\gamy)$ and $\kappa\in(0,\gamy-\gamma)$.

Let $\varphi$ be an admissible test function.

For all $\alpha\in(0,\aly)$, there exist $C_{\alpha,\gamma}(T,x_0,y_0,Q)\in(0,\infty)$ and $C_{\alpha,\gamma}(T,x_0,y_0,Q,\varphi)\in(0,\infty)$, such that for all $\epsilon\in(0,1)$ and $\delta>0$,
\[
\underset{t\in[0,T]}\sup~\bigl(\E|X_\delta^\epsilon(t)-\Xb_\delta(t)|_{L^2}^2\bigr)^{\frac12}\le C_{\alpha,\gamma}(T,x_0,y_0,Q)\delta^{-\frac{1-\alpha-\gamma}{2}}\epsilon^{\frac12}
\]
and
\[
\underset{t\in[0,T]}\sup~|\E[\varphi(X_\delta^\epsilon(t))]-\E[\varphi(\Xb_\delta(t))]|\le C_{\alpha,\gamma}(T,x_0,y_0,Q,\varphi)\delta^{-(1-\alpha-\gamma)}\epsilon.
\]
\end{lemma}

Second, the distances between $X_\delta^\epsilon$ and $X^\epsilon$, and between $\Xb_\delta$ and $\Xb$, are estimated in the following result, using standard arguments.
\begin{lemma}\label{lem:general2}
Let Assumption~\ref{ass:general} be satisfied. Let $T\in(0,\infty)$, and assume that $x_0\in L^2$ and $y_0\in L^2$. Let $\varphi$ be an admissible test function. Let $\alpha\in[0,\aly)$. There exist $C_\alpha(T,x_0,y_0,Q)\in(0,\infty)$ and $C_\alpha(T,x_0,y_0,Q,\varphi)\in(0,\infty)$ such that for all $\epsilon\in(0,1)$ and $\delta>0$, one has 
\[
\underset{t\in[0,T]}\sup~\bigl(\E|X_\delta^\epsilon(t)-X^\epsilon(t)|_{L^2}^2\bigr)^{\frac12}+\underset{t\in[0,T]}\sup~\bigl(\E|\Xb_\delta(t)-\Xb(t)|_{L^2}^2\bigr)^{\frac12}\le C_{\alpha}(T,x_0,y_0,Q)\delta^\alpha.
\]
and
\[
\underset{t\in[0,T]}\sup~|\E[\varphi(X_\delta^\epsilon(t))]-\E[\varphi(X^\epsilon(t))]|+\underset{t\in[0,T]}\sup~|\E[\varphi(\Xb_\delta(t))]-\E[\varphi(\Xb(t))]|\le C_\alpha(T,x_0,y_0,Q,\varphi)\delta^{2\alpha}.
\]

\end{lemma}

\begin{proof}[Proof of Lemma~\ref{lem:general1}]
This is a straightforward application of Theorems~\ref{th:strong_regular} and~\ref{th:weak_regular} combined with Lemma~\ref{lem:regul}.
\end{proof}

\begin{proof}[Proof of Lemma~\ref{lem:general2}]
Consider first the estimates of the strong error. Since the nonlinear operators $F$ and $\Fb$ are globally Lispchitz continuous, it is sufficient to prove the following estimate:
\begin{align*}
\E\big|\int_{0}^{t}e^{(t-s)A}\bigl(e^{\delta A}-I\bigr)dW^Q(s)|_{L^2}^2 ds&=\int_{0}^{t}\|e^{sA}\bigl(e^{\delta A}-I\bigr)Q^{\frac12}\|_{\mathcal{R}(L^2,L^2)}^2ds\\
&\le \|(e^{\delta A}-I)(-A)^{-\alpha}\|_{\mathcal{L}(L^2,L^2)}^2\int_{0}^{t}\|e^{sA}(-A)^{\alpha}Q^{\frac12}\|_{\mathcal{R}(L^2,L^2)}^2ds\\
& \le C_\alpha\delta^{2\alpha}M_{\alpha,2}(Q,T)^2,
\end{align*}
and to the strong error estimates are straightforward consequences of the Gronwall Lemma.

It remains to prove the estimates of the weak error. Since the argument is the same for both estimates, we only deal with the second one. Note that
\[
\E[\varphi(\Xb_\delta(t))]-\E[\varphi(\Xb(t))]=\E[\ub(0,\Xb_\delta(t))]-\E[\ub(t,\Xb_\delta(0))],
\]
where $\ub$ is defined by the expression~\eqref{eq:ub}. Observe that, even if Assumption~\ref{ass:general} is satisfied instead of Assumption~\ref{ass:regular}, the regularity estimates on spatial derivatives of $\ub$ stated in Proposition~\ref{propo:ub} remain valid without modification.

Using It\^o formula, one obtains
\begin{align*}
\E[\varphi(\Xb_\delta(t))]&-\E[\varphi(\Xb(t))]\\
&=\E\int_{0}^{t}\sum_{n\in\N}q_n\Bigl(D^2\ub(t-s,\Xb_\delta(s)).(e^{\delta A}f_n,e^{\delta A}f_n)-D^2\ub(t-sn,\Xb_\delta(s)).(f_n,f_n)\Bigr)ds\\
&=\E\int_{0}^{t}\Bigl({\rm Tr}\bigl(D^2\ub(t-s,\Xb_\delta(s))e^{\delta A}Qe^{\delta A}\bigr)-{\rm Tr}\bigl(D^2\ub(t-s,\Xb_\delta(s))Q\bigr)\Bigr)ds\\
&=\E\int_{0}^{t}{\rm Tr}\bigl(D^2\ub(t-s,\Xb_\delta(s))\bigl(e^{\delta A}-I\bigr)Qe^{\delta A}\bigr)ds\\
&~+\E\int_{0}^{t}{\rm Tr}\bigl(D^2\ub(t-s,\Xb_\delta(s))Q\bigl(e^{\delta A}-I\bigr)\bigr)ds,
\end{align*}
where $D^2\ub(t,x)$ is interpreted as a bounded, self-adjoint, linear operator from $L^2$ to $L^2$, instead of a symmetric, bilinear form on $L^2$, using Riesz Theorem: for all $h\in L^2$, $D^2\ub(t,x).h\in L^2$ is caracterized by
\[
\langle D^2\ub(t,x)h,\cdot\rangle=D^2\ub(t,x).(h,\cdot).
\]
Let $\alpha\in(0,\aly)$ and $\kappa\in(0,\aly-\alpha)$. Then, using the H\"older type inequality for Schatten norms, for all $0\le s< t\le T$,
\begin{align*}
\big|{\rm Tr}\bigl(D^2\ub(t-s,\Xb_\delta(s))&\bigl(e^{\delta A}-I\bigr)Qe^{\delta A}\bigr)\big|=\|D^2\ub(t-s,\Xb_\delta(s))\bigl(e^{\delta A}-I\bigr)Qe^{\delta A}\|_{\mathcal{L}_1(L^2)}\\
&\le \|D^2\ub(t-s,\Xb_\delta(s))(-A)^{1-2\kappa}\|_{\mathcal{L}_\infty(L^2)}\|(-A)^{-1+2\kappa}(I-e^{\delta A})\|_{\mathcal{L}_{\varsigma}(L^2)}\|Q\|_{\mathcal{L}_{\frac{\varrho}{2}}(L^2)},
\end{align*}
where $1=\frac{2}{\varrho}+\frac{1}{\varsigma}$. By assumption, one has $\|Q\|_{\mathcal{L}_{\frac{\varrho}{2}}}<\infty$. In addition, thanks to Proposition~\ref{propo:ub}, one has
\[
\|D^2\ub(t-s,\Xb_\delta(s))(-A)^{1-2\kappa}\|_{\mathcal{L}_\infty(L^2)}=\|D^2\ub(t-s,\Xb_\delta(s))(-A)^{1-2\kappa}\|_{\mathcal{L}(L^2,L^2)}\le C_\kappa(t-s)^{-1+2\kappa}.
\]
Finally, $(-A)^{-1+2\kappa}(I-e^{\delta A})$ is a self-adjoint, compact, linear operator, thus, for $\alpha\le \frac12$, one has
\begin{align*}
\|(-A)^{-1+2\kappa}(I-e^{\delta A})\|_{\mathcal{L}_\varsigma(L^2)}^\varsigma&=\sum_{n\in\N}\lambda_n^{-(1-2\kappa)\varsigma}(1-e^{-\delta\lambda_n})^{\varsigma}\\
&\le C_\alpha \delta^{2\alpha\varsigma}\sum_{n\in\N}\lambda_n^{-(1-2\kappa-2\alpha)\varsigma}.
\end{align*}
Finally, with the condition $\alpha+\kappa<\aly=\frac{1}{2}(1-\frac{d}{2\varsigma})$, one has $(1-2\kappa-2\alpha)\varsigma>\frac{d}{2}$, thus $\sum_{n\in\N}\lambda_n^{-(1-2\kappa-2\alpha)\varsigma}<\infty$.

Finally, one obtains
\[
\big|{\rm Tr}\bigl(D^2\ub(t-s,\Xb_\delta(s))\bigl(e^{\delta A}-I\bigr)Qe^{\delta A}\bigr)\big|\le C_\alpha \delta^{2\alpha}(t-s)^{-1+2\kappa},
\]
and similarly
\[
\big|{\rm Tr}\bigl(D^2\ub(t-s,\Xb_\delta(s))Q\bigl(e^{\delta A}-I\bigr)\bigr)\big|\le C_\alpha \delta^{2\alpha}(t-s)^{-1+2\kappa}.
\]

It is then straightforward to conclude that
\[
\big|\E[\varphi(\Xb_\delta(t))]-\E[\varphi(\Xb(t))]\big|\le C_\alpha \delta^{2\alpha}.
\]
This concludes the proof of Lemma~\ref{lem:general2}.
\end{proof}

We are now in position to provide the proof of Theorem~\ref{th:general}, which consists in choosing $\delta$ in terms of $\epsilon$ to maximize the order of convergence.
\begin{proof}[Proof of Theorem~\ref{th:general}]
Thanks to the strong and weak error estimates from Lemmas~\ref{lem:general1} and~\ref{lem:general2}, one obtains, for all $\epsilon\in(0,1)$ and $\delta>0$,
\begin{align*}
\underset{t\in[0,T]}\sup~\bigl(\E|X^\epsilon(t)-\Xb(t)|_{L^2}^2\bigr)^{\frac12}&\le C_{\alpha,\gamma}(T,x_0,y_0,Q)\Bigl(\delta^{-\frac{1-\alpha-\gamma}{2}}\epsilon^{\frac12}+\delta^{\alpha}\Bigr),\\
\underset{t\in[0,T]}\sup~|\E[\varphi(X^\epsilon(t))]-\E[\varphi(\Xb(t))]|&\le C_{\alpha,\gamma}(T,x_0,y_0,Q,\varphi)\Bigl(\delta^{-(1-\alpha-\gamma)}\epsilon+\delta^{2\alpha}\Bigr).
\end{align*}
Choosing $\delta=\epsilon^{\frac{1}{1+\alpha-\gamma}}$, with $\aly-\alpha$ and $\gamy-\gamma$ arbitrarily small then concludes the proof.
\end{proof}

\begin{rem}
Let us replace Assumption~\ref{ass:general} by the following condition: $\aly\in[0,1)$ is such that for all $\alpha\in[0,\aly$ and all $p\ge 2$, one has $\|(-A)^{\alpha-\frac12}Q^{\frac12}\|_{\mathcal{R}(L^2,L^p)}<\infty$. Then the results of this section can be generalized as follows, using similar techniques. Lemma~\ref{lem:general1} holds true, whereas Lemma~\ref{lem:general2} needs to be modified: the strong error remains bounded by $C_\alpha\delta^\alpha$, and the weak error is bounded by $C_\alpha \delta^{\min(1,2\alpha)}$. On the one hand, if $\aly\ge\frac12$, the situation is the same as in Theorem~\ref{th:general}. On the other hand, if $\aly\ge\frac12$, the strong and the weak rates one obtains using the approximation approach considered above, are $\frac{\aly}{1+\aly-\gamy}$ and $\frac{1}{2-\aly-\gamy}$ respectively. This statement and the approach are not satisfactory in this case since the weak order is not equal to twice the strong order anymore. Whether this issue can be fixed, and whether the rates of convergence given above are optimal, is left for future works.
\end{rem}

\section{Efficient numerical approximation of the slow component}\label{sec:hmm}

The goal of this section is to describe a temporal discretization scheme for the slow component~$X^\epsilon$ in~\eqref{eq:SPDE}, which is stable and efficient when $\epsilon\to 0$. Indeed, given a time-step size $h>0$, stability for the discretization of the evolutions of $X^\epsilon$ and $Y^\epsilon$ requires to choose $h$ such that $h={\rm O}(\epsilon)$. The scheme proposed below is based on Heterogeneous Multiscale Methods, see~\cite{B:2013} and references therein. Instead of using a single time-step size $h>0$, two time-step sizes $\Delta t>0$ and $\delta t>0$ are introduced. The slow component $X^\epsilon$ is discretized using a macro-scheme, with time-step size $\Delta t$: the scheme is constructed such that $\Delta t$ does not depend on the small parameter $\epsilon$. The fast component $Y^\epsilon$ is discretized using a micro-scheme, with time-step size $\delta t$. Since in~\eqref{eq:SPDE}, the fast component $Y^\epsilon(t)=Y(\epsilon^{-1}t)$ is not coupled with the slow component $X^\epsilon$, in this manuscript we can rely on a discretization of the process $Y$, with a time-step size $\tau>0$.

The detailed construction of the scheme is presented and discussed in Section~\ref{sec:hmm1}. Convergence results are stated in Section~\ref{sec:hmm2}. The proofs are omitted, since they would be similar to those in~\cite{B:2013}, where the slow component was not driven by a stochastic forcing.

\subsection{Construction of the scheme}\label{sec:hmm1}

As explained above, the main parameters of the multiscale scheme are the macro-time step size $\Delta t>0$ and the micro-time step size $\tau>0$. Two other integer parameters $M$ and $M_a\in\left\{1,\ldots,M\right\}$ are used, to insert data from the micro-scheme into the macro-scheme, in terms of temporal averages.

In this section, to avoid cumbersome notation, precise regularity conditions, and dependence in error estimates, on the initial conditions $x_0,y_0$ are not indicated.

\subsubsection{Micro-scheme}

Let $\bigl(Y_m^\tau\bigr)_{m\in\N_0}$ be computed using a numerical integrator $\Phi^\tau$ for the stochastic process $\bigl(Y(t)\bigr)_{t\ge 0}$. It is assumed that this discrete-time process defines an ergodic Markov chain on $L^2$, with unique invariant probability distribution denoted by $\mu^\tau$.

It is natural, see for instance~\cite{B:2013}, to assume that the error between $\mu$ and $\mu^\tau$ is of the order $\tau^{2\gamma}$, for all $\gamma\in[0,\gamy)$, in the following sense: for all functions $\varphi:L^2\to \mathbb{R}$ of class $\mathcal{C}^2$, with bounded first and second order derivatives, and all $\gamma\in[0,\gamy)$, there exists $C_\gamma(\varphi)\in(0,\infty)$ such that
\begin{equation}\label{eq:errtau}
|\int \varphi d\mu^\tau-\int \varphi d\mu|\le C_\gamma(\varphi)\tau^{2\gamma}.
\end{equation}
Moreover, define an averaged coefficient $\Fb^\tau$ with respect to the probability distribution $\mu^\tau$:
\begin{equation}\label{eq:Fbtau}
\Fb^\tau(x)=\int F(x,y) d\mu^\tau(y),~\forall x\in L^2.
\end{equation}
The approximation result above is extended as follows: assume that, for all $\gamma\in[0,\gamy)$, there exists $C_\gamma\in(0,\infty)$, such that
\[
\underset{x\in L^2}\sup~|\Fb(x)-\Fb^\tau(x)|_{L^2}\le C_\gamma \tau^{2\gamma}.
\]
The conditions above are satisfied if $Y$ is the solution of a parabolic semilinear SPDE, driven by additive noise, and $Y^\tau$ is obtained applying the linear-implicit Euler scheme, under suitable conditions on the nonlinearity in the equation, see~\cite{B:2014}.

To state convergence results, notation concerning the speed of convergence to equilibrium is introduced. Let $\rho:(0,\infty)\to (0,\infty)$ be non-increasing, $\rho(t)\underset{t\to\infty}\to 0$, and assume that
\begin{equation}\label{eq:speed}
|\E[\varphi(Y_m^\tau)]-\int \varphi d\mu^\tau|\le C(\varphi)\rho(m\tau),
\end{equation}
and that
\[
\underset{x\in L^2}\sup~\big|\E[F(x,Y_m^\tau)]-\Fb^\tau(x)\big|_{L^2}\le C\rho(m\tau).
\]

Finally, define the following quantities
\[
R_1(M,M_a,\tau)=\frac{1}{M_a}\sum_{m=M-M_a+1}^{M}\rho(m\tau)~,~R_2(M,M_a,\tau)=\frac{1}{M_a^2}\sum_{M-M_a+1\le m_1<m_2\le M}\rho((m_2-m_1)\tau).
\]
Using a Cesaro type argument, for fixed $\tau>0$, if $M_a\to\infty$, then both $R_1(M,M_a,\tau)\to 0$ and $R_2(M,M_a,\tau)\to 0$. If $M-M_a\to \infty$, then also $R_1(M,M_a,\tau)\to 0$.

More precisely, if $\rho(t)=e^{-ct}$ for some $c>0$, note that there exists $C\in(0,\infty)$ such that for all $M_a\le M$ and all $\tau>0$,
\[
R_1(M,M_a,\tau)\le \frac{Ce^{-c(M-M_a+1)\tau}}{M_a\tau+1}~,~  R_2(M,M_a,\tau)\le \frac{C}{M_a\tau+1}.
\]

\subsubsection{Macro-scheme}

We are now in position to define the macro-scheme. The principle is to approximate $\Xb(t)$ instead of $X^\epsilon(t)$, thanks to the averaging principle, and using error estimates given by Theorems~\ref{th:strong_regular},~\ref{th:weak_regular} and~\ref{th:general}. It is thus sufficient to compute an approximation of the averaged coefficient $\Fb$, using the ergodicity of the micro-scheme and the error estimate~\eqref{eq:errtau}, and to apply a standard integrator with time-step size $\Delta t>0$.

Set $Y_{n,m}^\tau=Y_{nM_t+m}^\tau$ for all $n\in\mathbb{N}$, and $m\in\left\{0,\ldots,M\right\}$. The macro-scheme is based on the linear implicit Euler scheme: define
\begin{equation}\label{eq:macro}
X_{n+1}=S_{\Delta t}(X_n+\Delta t\tilde{F}_n+\Delta W_n^Q\bigr),
\end{equation}
where $X_0^\epsilon=x_0$, $S_{\Delta t}=(I-\Delta tA)^{-1}$, $\Delta W_n^Q=W^Q\bigl((n+1)\Delta\bigr)-W^Q\bigl(n\Delta t\bigr)$ are Wiener increments, and with the following approximation of the nonlinearity,
\begin{equation}\label{eq:average_micromacro}
\tilde{F}_n=\frac{1}{M_a}\sum_{m=M-M_a+1}^{M}F(X_n,Y_{n,m}^\tau),
\end{equation}
computed as a temporal average, depending on the parameters $M_a$ and $M$.

\subsection{Convergence of the multiscale scheme \eqref{eq:macro}-\eqref{eq:average_micromacro}}\label{sec:hmm2}

\subsubsection{Auxiliary schemes}

In order to analyze the multi-scale scheme given by~\eqref{eq:macro}-\eqref{eq:average_micromacro}, and to give a clear discussion, several schemes are introduced.

First, applying the same integrator as in~\eqref{eq:macro}, {\it i.e.} the linear implicit Euler scheme, to discretize the averaged SPDE~\eqref{eq:av}, introduce the following scheme,
\[
\Xb_{n+1}=S_{\Delta t}\bigl(\Xb_n+\Delta t\Fb(\Xb_n)+\Delta W_n^Q\bigr),~\quad \Xb_0=x_0.
\]

Second, due to the error in sampling the invariant distribution $\mu$ using the micro-scheme with time-step size $\tau>0$, define a modified averaged equation
\[
d\Xb^\tau(t)=A\Xb^\tau(t)dt+\Fb^\tau(\Xb^\tau(t))dt+dW^Q(t),~\Xb^\tau(0)=x_0,
\]
and the associated numerical discretization
\begin{equation}\label{eq:scheme_tau}
\Xb_{n+1}^\tau=S_{\Delta t}\bigl(\Xb_n^\tau+\Delta t\Fb^\tau(\Xb_n^\tau)+\Delta W_n^Q\bigr),~\quad \Xb_0^\tau=x_0.
\end{equation}

Based on the literature concerning the numerical analysis of SPDEs, rates of convergence for these numerical schemes are assumed to be as follows: for all $T\in(0,\infty)$, all $\alpha\in[0,\aly)$, and all test functions $\varphi$ of class $\mathcal{C}_b^2$,
\begin{equation}\label{eq:errorDeltat}
\begin{aligned}
\underset{0\le t\le T}\sup~\bigl(\E|\Xb(n\Delta t)-\Xb_n|^2\bigr)^{\frac12}&\le C_\alpha(T)\Delta t^{\min(\alpha,\frac12)},\\
\underset{0\le t\le T}\sup~\big|\E[\varphi(\Xb(n\Delta t))]-\E[\varphi(\Xb_n)]|\big|&\le C_\alpha(T,\varphi)\Delta t^{\min(2\alpha,1)}.
\end{aligned}
\end{equation}

\subsubsection{Error estimates}

Proposition~\ref{propo:hmm} below states a general convergence result, depending on the parameters $\Delta t$, $\tau$, $M$ and $M_a$.

Let $\beta_{\rm max}=\frac12$ when Assumption~\ref{ass:regular} is satisfied, and recall that $\beta_{\rm max}=\frac{\aly}{1+\aly-\gamy}$ if Assumption~\ref{ass:general} is satisfied.
\begin{propo}\label{propo:hmm}
For all $T\in(0,\infty)$, all $\alpha\in[0,\aly)$, $\gamma\in[0,\gamy)$, and $\beta\in[0,\beta_{\rm max})$, there exists $C_{\alpha,\gamma,\beta}(T)\in(0,\infty)$ such that the strong error is of size
\begin{equation}\label{eq:propo_hmm_1}
\begin{aligned}
\underset{0\le n\Delta t\le T}\sup~\bigl(\E|X_n-X^\epsilon(n\Delta t)|_{L^2}^{2}\bigr)^{\frac12}&\le C_{\alpha,\gamma,\beta}(T)\bigl(\epsilon^\beta+\Delta t^{\min(\alpha,\frac12)}+\tau^{2\gamma}\bigr)\\
&+C_{\alpha,\gamma,\beta}(T)\Bigl( \sqrt{R_1(M,M_a,\tau)}+\sqrt{\Delta t}\bigl(\frac{1}{\sqrt{M_a}}+\sqrt{R_2(M,M_a,\tau)}\bigr)\Bigr),
\end{aligned}
\end{equation}
and, for all test functions $\varphi$ of class $\mathcal{C}_b^2$, there exists $C_{\alpha,\gamma,\beta}(T,\varphi)\in(0,\infty)$ such that the weak error is of size
\begin{equation}\label{eq:propo_hmm_2}
\begin{aligned}
\underset{0\le n\Delta t\le T}\sup~\big|\E[\varphi(X_n)]-\E[\varphi(X^\epsilon(n\Delta t))]\big|&\le C_{\alpha,\gamma,\beta}(T,\varphi)\bigl(\epsilon^{2\beta}+\Delta t^{\min(2\alpha,1)}+\tau^{2\gamma}\bigr)\\
&~+C_{\alpha,\gamma,\beta}(T,\varphi)\Bigl(R_1(M,M_a,\tau)+\Delta t\bigl(\frac{1}{M_a}+R_2(M,M_a,\tau)\bigr)\Bigr).
\end{aligned}
\end{equation}
In addition, for all test functions $\psi$ of class $\mathcal{C}_b^2$, there exists $C_{\gamma}(\psi)\in(0,\infty)$, such that,
\begin{equation}\label{eq:propo_hmm:3}
\underset{n\in\N}\sup~|\E[\psi(Y_{nM})]-\int \psi d\mu|\le C_\tau(\psi)\Bigl(\tau^{2\gamma}+\rho(nM\tau)\Bigr).
\end{equation}
\end{propo}

In fact, Proposition~\ref{propo:hmm} is a straightforward corollary of Proposition~\ref{propo:hmm_bis} below, combined with results stated above:
\begin{itemize}
\item Theorems~\ref{th:strong_regular} and~\ref{th:weak_regular}, or Theorem~\ref{th:general}, to deal with the error in the averaging principle, which are the main results of this article,
\item strong and weak error estimates~\eqref{eq:errorDeltat} for the macro-scheme applied to the averaged SPDE~\eqref{eq:av},
\item the sampling error~\eqref{eq:errtau} between the invariant distributions $\mu$ and $\mu^{\tau}$, which gives an error estimate $\underset{0\le n\Delta t\le T}\sup~\bigl(\E|\Xb_n-\Xb_n^\tau|^2\bigr)^{\frac12}\le C_\gamma \tau^{2\gamma}$ by a straightforward Gronwall type argument.
\end{itemize}
Note that~\eqref{eq:propo_hmm:3} is a straightforward consequence of~\eqref{eq:errtau} and~\eqref{eq:speed}.

\begin{propo}\label{propo:hmm_bis}
For all $T\in(0,\infty)$, there exists $C(T)\in(0,\infty)$ such that, for all $\Delta t\in(0,1)$, $\tau\in(0,1)$, and $1\le M_a\le M$, one has
\begin{equation}\label{eq:propo_hmm_1_bis}
\underset{0\le n\Delta t\le T}\sup~\bigl(\E|X_n-X_n^\tau|_{L^2}^{2}\bigr)^{\frac12}\le C(T)\Bigl( \sqrt{R_1(M,M_a,\tau)}+\sqrt{\Delta t}\bigl(\frac{1}{\sqrt{M_a}}+\sqrt{R_2(M,M_a,\tau)}\bigr)\Bigr),
\end{equation}
and, for all test functions $\varphi$ of class $\mathcal{C}_b^2$, there exists $C(T,\varphi)\in(0,\infty)$ such that, for all $\Delta t\in(0,1)$, $\tau\in(0,1)$, and $1\le M_a\le M$, one has
\begin{equation}\label{eq:propo_hmm_2_bis}
\underset{0\le n\Delta t\le T}\sup~\big|\E[\varphi(X_n)]-\E[\varphi(\Xb_n^\tau)]\big|\le C(T,\varphi)\Bigl(R_1(M,M_a,\tau)+\Delta t\bigl(\frac{1}{M_a}+R_2(M,M_a,\tau)\bigr)\Bigr).
\end{equation}
\end{propo}

Observe that Proposition~\ref{propo:hmm_bis} implies the convergence of the macro-scheme~\eqref{eq:macro} to the scheme~\eqref{eq:scheme_tau}, when $M_a\to \infty$, for any fixed values of $\Delta t$ and $\tau$. Note that to respect time-scales in~\eqref{eq:SPDE}, it is appropriate to choose parameters such that $M\tau=\epsilon^{-1}\Delta t$, and also $M_a\tau=\tilde{M}_a\epsilon^{-1}\Delta t$, thus the convergence property stated above may be interpreted as arising from taking the limit $\epsilon\to 0$. The limit scheme~\eqref{eq:scheme_tau} is not an integrator for the averaged equation~\eqref{eq:av}, but to a modified equation, with a residual depending on the micro time-step size $\tau$.

A full analysis of the cost of the multiscale scheme~\eqref{eq:macro}-\eqref{eq:average_micromacro}, depending on parameters $\Delta t$, $\tau$, $M$ and $M_a$, requires to balance the error terms in~\eqref{eq:propo_hmm_1} and~\eqref{eq:propo_hmm_2}. We refer to~\cite{B:2013}.

To conclude this section, a skecth of proof of Proposition~\ref{propo:hmm_bis} is provided, see~\cite{B:2013} for more details.

\begin{proof}[Sketch of proof of Proposition~\ref{propo:hmm_bis}]
To deal with the strong error estimate~\eqref{eq:propo_hmm_1_bis}, note that
\begin{align*}
X_n-\Xb_n^\tau&=\Delta t\sum_{k=0}^{n-1}S_{\Delta t}^{n-k}\bigl(\tilde{F}_k-\Fb^\tau(\Xb_k^\tau)\\
&=\Delta t\sum_{k=0}^{n-1}S_{\Delta t}^{n-k}\bigl(\Fb^\tau(X_k)-\Fb^\tau(\Xb_k^\tau)+\Delta t\sum_{k=0}^{n-1}S_{\Delta t}^{n-k}\bigl(\tilde{F}_k-\Fb^\tau(X_k^\tau).
\end{align*}
On the one hand, thanks to the Lipschitz continuity of $\Fb^\tau$, one has
\[
\bigl(\E\big|\Delta t\sum_{k=0}^{n-1}S_{\Delta t}^{n-k}\bigl(\Fb^\tau(X_k)-\Fb^\tau(\Xb_k^\tau)\big|_{L^2}^2\bigr)^ {\frac12}\le C\Delta t\sum_{k=0}^{n-1}\bigl(\E|X_k-\Xb_k^\tau|_{L^2}^2\bigr)^{\frac12}.
\]
On the other hand, a straightforward expansion yields
\begin{align*}
\E\big|\Delta t\sum_{k=0}^{n-1}S_{\Delta t}^{n-k}&\bigl(\tilde{F}_k-\Fb^\tau(X_k^\tau)|\big|^2\le\Delta t^2\sum_{k=0}^{n-1}\E|\tilde{F}_k-\Fb^\tau(X_k^\tau)|^2\\
&~+2\Delta t^2\sum_{0\le k_1<k_2\le n-1}\E\langle S_{\Delta t}^{n-k_1}\bigl(\tilde{F}_{k_1}-\Fb^\tau(X_{k_1}^\tau),S_{\Delta t}^{n-k_2}\bigl(\tilde{F}_{k_2}-\Fb^\tau(X_{k_2}^\tau)\rangle\\
&=\mathcal{E}_1+\mathcal{E}_2.
\end{align*}

An expansion of the average which defines $\tilde{F}_k$ yields
\begin{align*}
\mathcal{E}_1&=\frac{\Delta t^2}{M_a^2}\sum_{k=0}^{n-1}\sum_{m=M-M_a+1}^{M}\E|F(X_k,Y_{k,m}^\tau)-\Fb^\tau(X_k)|^2\\
&~+\frac{2\Delta t^2}{M_a^2}\sum_{k=0}^{n-1}\sum_{M-M_a+1\le m_1<m_2\le M}\E\langle F(X_k,Y_{k,m_1}^\tau)-\Fb^\tau(X_k),F(X_k,Y_{k,m_2}^\tau)-\Fb^\tau(X_k)\rangle\\
&\le C\frac{\Delta t}{M_a}+C\frac{\Delta t}{M_a^2}\sum_{M-M_a+1\le m_1<m_2\le M}\rho((m_2-m_1)\tau),
\end{align*}
thanks to a conditioning argument. In addition, using another conditioning argument, one gets
\[
|\mathcal{E}_2|\le C\Delta t^2\sum_{0\le k_1<k_2\le n-1}\frac{1}{M_a}\sum_{m=M-M_a+1}^{M}\rho(m\tau)\le CR_1(M,M_a,\tau).
\]
It remains to apply a discrete Gronwall Lemma to conclude the proof of the strong error estimate.

The treatment of the weak error estimate~\eqref{eq:propo_hmm_2_bis} requires to introduce the auxiliary function $\ub^{\tau}$ as follows: for all $n\in\N_0$ and $x\in L^2$,
\[
\ub^\tau(n,x)=\E[\varphi(\Xb_n^\tau) \big| \Xb_0^\tau=x].
\]
The weak error is then written as a telescoping sum
\begin{align*}
\E[\varphi(X_n)]-\E[\varphi(\Xb_n^\tau)]&=\E[\ub^\tau(0,X_n)]-\E[\ub^\tau(n,X_0)]\\
&=\sum_{k=0}^{n-1}\bigl(\E[\ub^\tau(n-k-1,X_{k+1})]-\E[\ub^\tau(n-k,X_k)]\bigr).
\end{align*}
In addition, using Markov property and a second-order Taylor expansion, one obtains
\begin{align*}
\E[\ub^\tau(n-k,X_k)]&-\E[\ub^\tau(n-k-1,X_{k+1})]\\
&=\E[\ub^\tau(n-k-1,S_{\Delta t}\bigl(X_k+\Delta t\Fb^\tau(X_k)+\Delta W_k^Q\bigr))]\\
&\hspace{3cm}-\E[\ub^\tau(n-k-1,S_{\Delta t}\bigl(X_k+\Delta t\tilde{F}_k+\Delta W_k^Q\bigr))]\\
&=\Delta t\E\bigl[D_x\ub^\tau(n-k-1,S_{\Delta t}\bigl(X_k+\Delta t\Fb^\tau(X_k)+\Delta W_k^Q\bigr)).\bigl(S_{\Delta t}\Fb^\tau(X_k)-S_{\Delta t}\tilde{F}_k\bigr)\bigr]\\
&+{\rm O}(\Delta t^2)\E|\Fb^\tau(X_k)-\tilde{F}_k|^2.
\end{align*}

Note that ${\rm O}(\Delta t^2)\E|\Fb^\tau(X_k)-\tilde{F}_k|^2$ is treated as $\mathcal{E}_1$ above.

The remaining error term is interpreted in terms of the auxiliary function
\[
\Psi(k,x,y)=-\E\bigl[D_x\ub^\tau(k,S_{\Delta t}\bigl(x+\Delta t\Fb^\tau(x)+\Delta W_0^Q\bigr).(S_{\Delta t}F(x,y))\bigr],
\]
as
\[
\frac{\Delta t}{M_a}\sum_{m=M-M_a+1}^{M}\bigl(\E[\Psi(n-k-1,X_k,Y_{k,m})]-\int\Psi(n-k-1,X_k,\cdot)d\mu^\tau\bigr).
\]
Studying regularity properties of $\Psi(n-k-1,x,\cdot)$, one then concludes using~\eqref{eq:speed}.

Finally,
\begin{align*}
\big|\E[\ub^\tau(n-k,X_k)]&-\E[\ub^\tau(n-k-1,X_{k+1})]\big|\le C\Delta tR_1(M,M_a,\tau)+C\Delta t^2\bigl(\frac{1}{M_a}+R_2(M,M_a,\tau)\bigr),
\end{align*}
and it remains to sum from $k=0$ to $k=n-1$ to conclude the proof of Proposition~\ref{propo:hmm_bis}.
\end{proof}

\section*{Ackowledgments}

The author would like to thank A.~Debussche for discussions during the preparation of this manuscript. This work was partially supported by the project BORDS (ANR-16-CE40-0027-01) operated by the French National Research Agency (ANR).


\end{document}